\documentclass[11pt]{article}
\pdfoutput=1
\usepackage{underscore}
\usepackage{array}
\newcolumntype{P}[1]{>{\centering\arraybackslash}p{#1}}
\newcolumntype{M}[1]{>{\centering\arraybackslash}m{#1}}
\usepackage[applemac]{inputenc}
\usepackage[numbers,compress]{natbib}
\usepackage{amsthm}
\usepackage{amsmath}
\usepackage{algcompatible}
\usepackage{algorithm}
\usepackage{booktabs}
\usepackage{enumerate}
\usepackage{graphicx}
\usepackage{amssymb}
\usepackage{latexsym}
\usepackage{epstopdf}
\usepackage{color}
\usepackage{xcolor}
\usepackage{bbm}
\usepackage{needspace}
\usepackage{color, colortbl}
\usepackage[english]{babel}
\usepackage{tikz}
\usepackage{caption}
\usepackage{subcaption}
\usetikzlibrary{arrows,automata}
\usetikzlibrary{positioning}
\usepackage{filecontents}
\usepackage{mathtools}
\usepackage{algorithm}
\usepackage{multirow}
\usepackage{algpseudocode}
\usepackage{appendix} 
\usepackage{fullpage}   

\usepackage{arydshln}

\DeclareCaptionFont{mysize}{\fontsize{8.0}{9.6}\selectfont}
\theoremstyle{plain}

\newtheorem{theorem}{Theorem}
\newtheorem{remark}{Remark}%
\newtheorem{lemma}{Lemma}
\newtheorem{corollary}{Corollary}
\newtheorem{definition}{Definition}%

\makeatother

\allowdisplaybreaks[4]  

\usepackage{babel}

\providecommand{\theoremname}{Theorem}

\usepackage{hyperref}
\usepackage{cleveref}

\usepackage{algorithm,algpseudocode,float}
\usepackage{lipsum}

\makeatletter
\newenvironment{breakablealgorithm}
  {
   \begin{center}
     \refstepcounter{algorithm}
     \hrule height.8pt depth0pt \kern2pt
     \renewcommand{\caption}[2][\relax]{
       {\raggedright\textbf{\ALG@name~\thealgorithm} ##2\par}%
       \ifx\relax##1\relax 
         \addcontentsline{loa}{algorithm}{\protect\numberline{\thealgorithm}##2}%
       \else 
         \addcontentsline{loa}{algorithm}{\protect\numberline{\thealgorithm}##1}%
       \fi
       \kern2pt\hrule\kern2pt
     }
  }{
     \kern2pt\hrule\relax
   \end{center}
  }
\makeatother

\begin{document}

\title{Riemannian Stochastic Hybrid Gradient Algorithm for Nonconvex Optimization}
\date{}
\author{Jiabao Yang}
\maketitle
\pagenumbering{arabic}

\abstract{In recent years, Riemannian stochastic gradient descent (R-SGD), Riemannian stochastic variance reduction (R-SVRG) and Riemannian stochastic recursive gradient (R-SRG) have attracted considerable attention on Riemannian optimization. Under normal circumstances, it is impossible to analyze the convergence of R-SRG algorithm alone. The main reason is that the conditional expectation of the descending direction is a biased estimation.  However, in this paper, we consider linear combination of three descent directions on Riemannian manifolds as the new descent direction (i.e., R-SRG, R-SVRG and R-SGD) and the parameters are time-varying. At first, we propose a Riemannian stochastic hybrid gradient(R-SHG) algorithm with adaptive parameters. The algorithm gets a global convergence analysis with a decaying step size. For the case of step-size is fixed, we consider two cases with the inner loop fixed and time-varying. Meanwhile, we quantitatively research the convergence speed of the algorithm. Since the global convergence of the R-SHG algorithm with adaptive parameters requires higher functional differentiability, we propose a R-SHG algorithm with time-varying parameters. And we obtain similar conclusions under weaker conditions.}

\maketitle
\section{Introduction}

Consider the following finite-sum optimization problems definition on a smooth Riemannian manifold $\mathcal{M}$
\begin{equation}
  \min\limits_{\omega \in \mathcal{M}} f(\omega):=\frac{1}{n}\sum_{i=1}^{n}f_i(\omega)\qquad (P)
\end{equation}
where the function $f_i:\mathcal{M}\rightarrow \mathcal{R},i=\{1,2,...,n\}$.

Problem (P) has many applications; including principal component analysis \cite{SatoH1,Balzano}, low rank matrix completion \cite{Boumal1,Kasai,Boumal2,Meyer}, Riemannian centroid computation \cite{Yuan}, independent component analysis \cite{Theis},dictionary learning \cite{Cherian,Sun} and so on.

Since some constrained optimization problems in Euclidean space can be converted to unconstrained problems on manifolds, it is interested in solving problem (P) over the Riemannian manifold space via Riemannian gradient methods. A common idea is that the negative of Riemannian gradient direction is used as the descent direction, that is calculate the Riemannian full gradient of function $f$: $\text{grad}f(\omega)=\frac{1}{n}\sum_{i=1}^{n}\text{grad}f_i(\omega)$, where the $\text{grad}f_i(\omega)$ denotes Riemannian gradient of the $i$th. If $n$ is large, the cost of computing and operating is expensive.

In Euclidean space, a popular choice to solve problem (P) is stochastic gradient descent(SGD) algorithm. Some scholars have achieved better results by improving robustness\cite{Staib}, adapting learning\cite{Kingma} etc. Inspired by the SGD algprithm in Euclidean space, other scholars have proposed the R-SGD algorithm on Riemannian manifold. Bonnabel proposed a R-SGD algorithm to extend SGD algorithm from Euclidean space to Riemannian manifold. However, we should point out that a popular choice is random selection of partial function gradients without taking the gradients of all functions. But R-SGD algorithm needs exponential mapping and parallel translation operation in each iteration. If these computational costs are lower than the computational Riemannian gradient, we can ignore them. Similar to the SGD algorithm in Euclidean space, when we use a large step size in R-SGD algorithm, the loss of training will decrease rapidly at first, but it may have a great influence around the solution. On the contrary, in order to obtain convergence, we require a large number of iterations when using smaller steps. Therefore, R-SGD algorithm can start with a large step size and gradually reduces the step size to avoid these problems. But due to the attenuation of step-size sequence, the convergence of R-SGD algorithm is slow.

In recent years, the technique of stochastic variance reduction have attracted considerable attention for minimizing the average of finite-number of loss functions. In Euclidean space, scholars prove that the method of variance reduction can accelerate SGD algorithm convergence \cite{Johnson}. The main idea is that by periodically calculating the gradient to correct the  deviation of stochastic gradient, and the gradient variance decreases with the progress of training. This leads to linear convergence.

Because of this, the paper \cite{Zhang, SatoH1} proposed a R-SVRG algorithm. Inspired by the variance reduction of non-convex optimization, \cite{SatoH1} has analyzed the R-SVRG algorithm of geodesic strongly convex function through a new theoretical analysis and explained the nonlinear (curve) geometric shape of the Riemannian manifold. This produce a linear convergence rate. The works are parallel with paper \cite{Zhang}. From the idea of paper \cite{Bonnabel}, paper \cite{SatoH1} proves the global convergence of the algorithm under retraction mapping and vector transport. But \cite{Zhang} is carried out under exponential mapping and parallel translation. It should emphasize that the local convergence rate is analyzed in \cite{SatoH1}. If the function $f$ is assumed to have global strong convexity in the search space, the global iterative complexity can be obtained.

Since R-SVRG algorithm uses double loop iteration, we need to add the condition that $\omega_0^s$ is transported to $\omega_t^s$, the vector transport of R-SVRG algorithm between the iterations of two distant points is required in the calculation. Its cost and difficulty will be improved. Therefore, a R-SRG algorithm independent of two distant points is proposed in \cite{SatoH2}, to avoid the calculation of contraction inverse and makes the calculation efficiency higher. The advantage of R-SRG algorithm over R-SVRG algorithm is more notable in the Riemannian than Euclidean case \cite{SatoH2}. In addition, from \cite{Cutkosky,Tran-Dinh}, Riemannian stochastic recursive momentum(R-SRM) algorithm is proposed in \cite{AndiH1}. The author considers the linear combination of R-SGD and R-SVRG, and obtained the R-SRM algorithm (the linear combination coefficient and step size of the algorithm are time-varying), and assumes that the optimization function is an unbiased estimation. It is proved that the expectation converges at the convergence rate of $\mathcal{O}(\frac{1}{T^{\frac{2}{3}}})$.

Because the calculation of exponential mapping and parallel translation are expensive, therefore, in this paper, we consider the situations with retraction mapping and vector transport. Inspired by \cite{AndiH1}, we consider the linear combination of three descent directions on Riemannian manifolds as the new descent direction (i.e., R-SRG, R-SVRG and R-SGD) and propose the two algorithms. And the linear combination of the parameters in the algorithms are time-varying.
Compared to the existing works, the key \textbf{contributions} of our paper are listed as follows

1) \,Commonly, the global convergence of R-SRG algorithm can not analyze alone as \cite{SatoH1}, the main reason is the conditional expectation of the descent direction is biased. In contrast, \cite{SatoH1} is unbiased. Therefore, we propose a R-SHG algorithm with adaptive parameters. In this way, the conditional expectation of the descent direction after the combination is still a biased estimation. For the case of reduced step size, by adapting the parameters of the R-SHG algorithm, we can get the global convergence. If special parameters are chosen, our results can be degenerated into \cite{SatoH1}. Moreover, for the case of fixed step size, we quantitatively research the convergence rate of the algorithm.

2) \,Research \cite{AndiH1} considers the linear combination of R-SRG and R-SGD. Our second algorithm (i,e, R-SHG algorithm with time-varying parameters) can obtain a faster convergence rate than them. If we consider the problem of expectation (online) minimization over Riemannian manifold $\mathcal{M}$, in that case, we can choose special parameters such that our results can be degenerated into \cite{AndiH1} and our results better. Moreover, we also give the convergence rate under fixed step size. These convergence conditions are weakly than the R-SHG algorithm with adaptive parameters.

3) \,Usually, choosing time-varying step size may accelerate the convergence of the algorithm.  In Riemannian manifold, the main consideration is to improve convergence speed by using time-varying step size.  However, our results imply that the convergence rate can also be accelerated under the condition of fixed step size by changing the parameters.

4) \,Convergence analysis(convergence rate) is complex in the algorithm, which is in itself a challenging problem. But our algorithm can do convergence analysis under time-varying step size and fixed step size. We use retraction mapping and vector transport, which is more general in Riemannian manifold than exponential mapping and parallel transport.

5) \,For the three special situations(i.e., the descent direction only use the R-SRG and R-SVRG term, the  retraction mapping and vector transport is taken as exponential mapping and parallel translation operations, and function $f$ is $\tau-$gradient dominated), we give the better conclusions.

The rest of the paper is organized as follows. Section 2 describes the Riemannian preliminaries and assumptions. Section 3 and 4 introduce the algorithm and prove the proposed algorithms' global convergence and local convergence rate. We also consider the convergence in several special cases. In Section 5, the conclusions and future research topics are given.

Notation and symbols: $\vert A\vert$: the cardinality of set A; $a_n=\mathcal{O}(b_n)$: $\lim\sup_{n \rightarrow \infty}\frac{a_n}{b_n}<\infty$

\section{Preliminaries and Assumption}

\subsection{Preliminaries}

A manifold whose tangent spaces are endowed with a smoothly varying inner product is called a Riemannian manifold. The smoothly varying inner product is called the Riemannian metric \cite{Absil}.
The inner product $g_x$: $T_x\mathcal{M} \times T_x\mathcal{M} \rightarrow \mathcal{R}$.
For the convenience of the following, we let $g_x(\eta_x,\zeta_x)=\langle\eta_x,\zeta_x\rangle_x=\langle\eta_x,\zeta_x\rangle$,
for any $\eta_x,\zeta_x \in T_x\mathcal{M}$. Let $\Vert \eta_x \Vert=\sqrt{\langle\eta_x,\eta_x\rangle}$.
$\mathbb{E}[\cdot \vert \mathcal{F}_t^s]$ denotes the conditional expectation
with respect to the random variable $I_t^s$, where
$\mathcal{F}_t^s=\sigma\{I_0^1,I_1^1,I_2^1,...,I_0^2,I_t^1,I_1^2,I_2^2,...,I_t^2,...,I_0^s,I_1^s,I_2^s,...,I_{t-1}^s\}$
is the $\sigma$-algebra and $I_0^s,s\in{1,2,...,S}$ is equal to complete set.
$\text{grad}f_{I_t^s}(\omega)=\frac{1}{\vert I_t^s \vert}\sum_{i\in I_t^s}\text{grad}f_i(\omega)$,
where $I_t^s\subset \{1,2,...,n\}$ is an index set with cardinality $\vert I_t^s \vert$.
The gradient $\text{grad}f(\omega)$ is defined as the unique element of $T_{\omega}\mathcal{M}$ that satisfies
\begin{equation*}
    Df(\omega)[\xi_{\omega}]=\langle\text{grad}f(\omega),\xi_{\omega}\rangle\qquad \xi_{\omega}\in T_{\omega}\mathcal{M}
\end{equation*}
where $Df(\omega): T_{\omega}\mathcal{M} \rightarrow \mathcal{R} $ is the derivative of $f$ at $\omega$.
The exponential map $Exp_x: T_x\mathcal{M} \rightarrow \mathcal{M}$ maps a tangent vector $\eta_x \in T_x\mathcal{M}$ along the geodesic leading
to $y=Exp_x(\eta_x)\in \mathcal{M}$ such that $\gamma(0)=x,\gamma(1)=y, \gamma'(0)=\eta_x$.
And the distant between $x$ and $y$ denotes $\text{dist}(x,y)=\Vert \eta_x\Vert$. For $\delta>0$, we denote $\mathbb{B}_x(0,\delta)=\{y\in \mathcal{M}\vert \text{dist}(x,y)\leq 1\}$.
In this paper, our analysis focus on retraction mapping and vector transport.
The definition of a retraction is as follows \cite{Absil}.
\begin{definition}
    $R:T\mathcal{M}\rightarrow \mathcal{M}$ is called a retraction on $\mathcal{M}$
    if the restriction $R_{\omega}:T_{\omega}\mathcal{M}\rightarrow \mathcal{M}$ to $T_{\omega}\mathcal{M}$
     for all $\omega \in \mathcal{M}$ satisfies:\\
1. $R_{\omega}(0_{\omega})=\omega $, where $0_{\omega}$ is the zero vector in $T_{\omega}\mathcal{M}$\\
2. $DR_{\omega}(0_{\omega})[\xi]=\xi$, for all $\xi \in T_{\omega}\mathcal{M}$
\end{definition}

Let $\Gamma_x^y$ be the parallel translation operator by the exponential mapping linking $x$ and $y$.
However, parallel translation sometimes computationally expensive, so we consider using vector transport replacing parallel translation.

\begin{definition}
   A vector transport on a mainfold $\mathcal{M}$ is a smooth mapping
   \begin{equation*}
       T\mathcal{M} \otimes T\mathcal{M} \rightarrow T\mathcal{M}:(\eta_{\omega},\xi_{\omega})\mapsto \mathcal{T}_{\eta_{\omega}}(\xi_{\omega}) \in T\mathcal{M}
   \end{equation*}
   satisfying the following diagram properties for all $x\in \mathcal{M}$\\
   1.$\mathcal{T}_{0_{{\omega}}}(\xi)=\xi$, where $\xi\in T_{\omega}\mathcal{M}, {\omega}\in\mathcal{M}$\\
   2.$\mathcal{T}_{\eta_{\omega}}(a\xi_{\omega}+b\theta_{\omega})=a\mathcal{T}_{\eta_{\omega}}(\xi_{\omega})+b\mathcal{T}_{\eta_{\omega}}(\theta_{\omega})$,
  where $a,b\in \mathcal{R}, \eta_x,\xi_x,\theta_x\in T_x\mathcal{M}$
\end{definition}
For the convenience of the following, we let $\mathcal{T}_{\omega}^{R_{\omega}(\eta_{\omega})}(\xi_{\omega})=\mathcal{T}_{\eta_{\omega}}(\xi_{\omega})$, for any $\eta_{\omega},\xi_{\omega} \in T_{\omega}(\mathcal{M})$.
Here, we further introduce the concept of $\tau-$gradient dominated function \cite{Polyak,Nesterov} which will also be used in this paper.
\begin{definition}
 We say function $f: $ $\mathcal{M}\rightarrow \mathcal{R}$ is $\tau-$gradient dominated, if for any $\omega \in \mathcal{M}$, we have $f(\omega)-f(\omega^*) \leq \tau \Vert \text{grad}f(\omega) \Vert^2$, where $\omega^*$ is a global minimizer of $f$.
\end{definition}

\subsection{Assumption}
In this article, we will use following assumptions.\\

\textbf{Assumption 1.a }Function $f$ and its component functions $f_i,i=1,2...n$ are continuously differentiable.\\

\textbf{Assumption 1.b }Function $f$ is thrice continuously differentiable, and its component functions $f_i,i=1,2...n$ are twice continuously differentiable.\\

\textbf{Assumption 2 }Iterate sequences produced by algorithms stay continuously in a compact neighbourhood $\Omega$, where the $\Omega$ is a neighbourhood around $\omega^\ast$.
Additionally, $\Omega$ is a $\rho$-totally retractive neighbourhood of $\omega^\ast$  where retraction $R$ is a diffeomorphism. And for all $t\geq 0,s\geq 1,\tau \in [0,1], R_{\omega_t^s}(-\tau \alpha_t^s V_t^s)\in \Omega$.

The $\rho$-totally neighborhood $\Omega$ of $\omega^\ast$ is a set such that for all $\omega\in \Omega$, $\Omega \subset R_{\omega}(0_{\omega}, \rho)$,
and $R_{\omega}(\cdot)$ is a diffeomorphism on $R_{\omega}(0_{\omega}, \rho)$. Assumption 1.b and 2 are basic for standard analysis.

\textbf{Assumption 3 }The sequence $\{\alpha_t^s\}$ of step sizes satisfies $\sum\limits_{s=1}^{\infty}\sum\limits_{t=0}^{m-1}\alpha_t^s=\infty$ and $\sum\limits_{s=1}^{\infty}\sum\limits_{t=0}^{m-1}(\alpha_t^s)^2<\infty$.

The conditions of assumption 3 are satisfied, for example, $\alpha_t^s=\frac{1}{t+s+1}$.\\

\textbf{Assumption 4 }The vector transport $\mathcal{T}$ is continuous and isometric on $\mathcal{M}$, i.e., for any $\omega\in \mathcal{M}$, $\eta,\xi,\zeta\in T_x\mathcal{M}$, satisfies $\langle \mathcal{T}_\eta\xi,\mathcal{T}_\eta\zeta\rangle=\langle\xi,\zeta\rangle$.

Similar to \cite{HuangW1,HuangW2}, we also can construct an isometric vector transport such that
assumption 4 holds.\\

\textbf{Assumption 5 }$\Gamma_y^x$ is the parallel transport operator from y to x, there exists a constant $M>0$, for any $x,y=R_x(\xi)\in\Omega$, satisfying
$\frac{1} {n}\sum_{i=1}^{n}\Vert\text{grad}f_i(x)-\Gamma_y^x\text{grad}f_i(y)\Vert^2\leq M^2\Vert\xi\Vert^2$.\\

\textbf{Assumption 6 \cite{HuangW1} } Difference between vector transport $\mathcal{T}_x^y$ and parallel transport $\Gamma_x^y$ associated with the same retraction $R$ is bounded.
There exists a constant $\theta>0$, for all $x,y=R_x(\xi)\in\Omega$ and $\eta\in T_xM$, satisfies $\Vert \Gamma_x^y\eta-\mathcal{T}_x^y\eta \Vert\leq \theta \Vert\xi\Vert \Vert\eta\Vert$.\\

\textbf{Assumption 7 }Function $f$ is retraction L-smooth with respect to retraction $R$.
There exists a constant $L>0$, for all $x,y=R_x(\xi)\in\Omega$, satisfies
$f(y)\leq f(x)+\langle\text{grad}f(x),\xi\rangle+\frac{L}{2}\Vert\xi\Vert^2$.\\

\textbf{Assumption 8 \cite{HuangW2} } There exists  $C_1, C_2>0$, and $\delta_{C_1,C_2}>0$, for any $x,y=R_x(\xi)\in\Omega$, when $\Vert \xi\Vert\leq \delta_{C_1,C_2}$, satisfies $\Vert\xi\Vert\leq C_1d(x,y)$, and $d(x,y)\leq C_2\Vert\xi\Vert$.

It is obvious that assumption 8 is a local property.\\

\section{Riemannian Stochastic Hybrid Gradient Algorithm with Adaptive Parameters}
In this section, firstly, we present the R-SHG algorithm with adaptive parameters.
For the case of reduced step size, we qualitatively analyze the convergence of the algorithm.
For the case of fixed step size, we quantitatively research the convergence of the algorithm.
Throughout this section, we consider two special cases.
Now, we propose the first algorithm.\\
\begin{breakablealgorithm}
    \caption{R-SHG algorithm with adaptive parameters}\label{Alg1}
    \begin{algorithmic}[1]
        \State\textbf {Input:} step size {$\alpha_t^s$}, frequency $m>0,0< \mu<1$, parameter $\psi_t^s$, $\phi_t^s$
        \State\textbf {Initialize:}  $\tilde{\omega}^0$;
         \For {for s=1,2,...,S}
        \State Caclulate the full Riemannian gradient $\text{grad}f(\tilde{\omega}^{s-1})$;
        \State Store $\omega_0^s=\tilde{\omega}^{s-1}$, $V_0^s=\text{grad}f(\tilde{\omega}_0^s)$, $\omega_1^s=R_{\omega_0^s}(-\alpha_0^sV_0^s)$
         \For {t=1,2,...,m-1 do}
        \State Choose $I_t^s\in\{1,2,...,n\}$ uniformly at random
        \State Caclulate the value of $\langle \mathcal{T}_{\omega_{t-1}^s}^{\omega_t^s} \big(V_{t-1}^s-\text{grad}f(\omega_{t-1}^s)\big),\text{grad}f(\omega_t^s)\rangle$
          \If{  $\langle \mathcal{T}_{\omega_{t-1}^s}^{\omega_t^s} \big(V_{t-1}^s-\text{grad}f(\omega_{t-1}^s)\big),\text{grad}f(\omega_t^s)\rangle\neq 0$ }
        \State $\tilde{\psi}_t^s=\min\{\psi_t^s,\frac{\mu\Vert \text{grad}f(\omega_t^s)\Vert^2}{\big\vert\langle \mathcal{T}_{\omega_{t-1}^s}^{\omega_t^s} (V_{t-1}^s-\text{grad}f(\omega_{t-1}^s)),\text{grad}f(\omega_t^s)\rangle\big\vert}\}$
          \ElsIf
        \State\,$\tilde{\psi}_t^s=\psi_t^s$
           \EndIf
        \State \begin{eqnarray}\label{Alg1.1}
            &V_t^s&=\phi_t^s\Big(\text{grad}f_{I_t^s}(\omega_t^s)-\mathcal{T}_{\omega_0^s}^{\omega_t^s} \big(\text{grad}f_{I_t^s}(\omega_0^s)-\text{grad}f(\omega_0^s)\big) \Big)\cr
            &&+\tilde{\psi}_t^s\Big(\text{grad}f_{I_t^s}(\omega_t^s)-\mathcal{T}_{\omega_{t-1}^s}^{\omega_t^s} \big(\text{grad}f_{I_t^s}(\omega_{t-1}^s)-V_{t-1}^s\big) \Big)\cr
            &&+(1-\phi_t^s-\tilde{\psi}_t^s)\text{grad}f_{I_t^s}(\omega_t^s)
            \end{eqnarray}
        \State $\omega_{t+1}^s=R_{\omega_t^s}(-\alpha_t^sV_t^s)$
         \EndFor
        \State $\tilde{\omega}^s=\omega_m^s$
         \EndFor
         \State We choose: option 1: ${\omega}_a=\tilde{\omega}^S$.
        \State We choose: option 2: ${\omega}_a$ uniformly randomly from $\{\{\omega_t^s\}_{t=0}^{m-1}\}_{s=1}^S$
         \State\textbf {Output:} ${\omega}_a$
\end{algorithmic}
\end{breakablealgorithm}

In the algorithm, we require that the parameters $\psi_t^s, \phi_t^s$ must be satisfies $0\leq \psi_t^s+\phi_t^s\leq 1$ and $\psi_t^s,\phi_t^s\geq 0$. In this paper, we always suppose the step size and the parameters are positive. By the definition of $\tilde{\psi}_t^s$, then we can get following inequalities.
\begin{lemma}\label{L0}
    Let $\tilde{\psi}_t^s$ be definited by the algorithm 1, then
    \begin{equation}\label{L0.1}
    -\mu\Vert \text{grad}f(\omega_t^s)\Vert^2\leq \tilde{\psi}_t^s\langle \mathcal{T}_{\omega_{t-1}^s}^{\omega_t^s} \big(V_{t-1}^s-\text{grad}f(\omega_{t-1}^s)\big),\text{grad}f(\omega_t^s)\rangle\leq \mu_t^s\Vert \text{grad}f(\omega_t^s)\Vert^2
    \end{equation}
    Especially, if $\psi_t^s\equiv 0$, then $\tilde{\psi}_t^s\equiv 0$, the inequalities also hold for any $0< \mu <1$.
\end{lemma}

This result is important in the convergence of the algorithm.
Normally, we can not determine the positive or negative number of $\langle \mathcal{T}_{\omega_{t-1}^s}^{\omega_t^s} \big(V_{t-1}^s-\text{grad}f(\omega_{t-1}^s)\big),\text{grad}f(\omega_t^s)\rangle$.
Therefore, we adjust $\tilde{\psi}_t^s$ to make $\tilde{\psi}_t^s\langle \mathcal{T}_{\omega_{t-1}^s}^{\omega_t^s} \big(V_{t-1}^s-\text{grad}f(\omega_{t-1}^s)\big),\text{grad}f(\omega_t^s)\rangle$ is bounded and sufficiently small.
The proof of lemma $\mathrm{\ref{L0}}$ is given in Appendix $\mathrm{\ref{secA}}$.

\subsection{Step size is reduced}

We first prove the following lemma of the estimator $V_t^s$.  This lamma palys an important role in this section.
\begin{lemma}\label{L1}
    Let $V_t^s$ be definited by algorithm 1, then,
    \begin{equation}\label{L1.1}
        \mathbb{E}[V_t^s \vert \mathcal{F} _t^s]=\text{grad} f(\omega_t^s)+\tilde{\psi}_t^s\mathcal{T}_{\omega_{t-1}^s}^{\omega_t^s} \big(V_{t-1}^s-\text{grad}f(\omega_{t-1}^s)\big)
    \end{equation}
    Furthermore, if $\psi_t^s\neq0$, then $V_t^s$ is a biased estimate.
\end{lemma}
\begin{proof}
    Suppose for any $\omega \in \Omega$ and $\omega$ is $\mathcal{F}_t^s$ measurable. By the definition of $I_t^s$ and $\mathcal{F}_t^s$, we get
  \begin{equation}\label{L1.2}
      \mathbb{E}[\text{grad}f_{I_t^s}(\omega) \vert \mathcal{F}_t^s]=\frac{1}{b}\cdot\frac{C_{n-1}^{b-1}}{C_n^b}(\sum\limits_{i=1}^n\text{grad}f_i(\omega))
      =\frac{1}{n}\sum\limits_{i=1}^n\text{grad}f_i(\omega)=\text{grad}f(\omega)
  \end{equation}
  Since $V_{t-1}^s$ is $\mathcal{F}_t^s$ measurable and $\mathcal{T}_{\omega_{t-1}^s}^{\omega_t^s}=\mathcal{T}_{\alpha_{t-1}^s V_{t-1}^s}$, we obtain that $\mathcal{T}_{\omega_{t-1}^s}^{\omega_t^s}$ is also $\mathcal{F}_t^s$ measurable.
  This together with \eqref{L1.2} gives
  \begin{eqnarray}\label{L1.3}
    &&\mathbb{E}[V_t^s \vert \mathcal{F} _t^s]\cr
    &=& \mathbb{E}[\phi_t^s\Big(\text{grad}f_{I_t^s}(\omega_t^s)-\mathcal{T}_{\omega_0^s}^{\omega_t^s} \big(\text{grad}f_{I_t^s}(\omega_0^s)-\text{grad}f(\omega_0^s)\big) \Big)\cr
    &&+\tilde{\psi}_t^s\Big(\text{grad}f_{I_t^s}(\omega_t^s)-\mathcal{T}_{\omega_{t-1}^s}^{\omega_t^s} \big(\text{grad}f_{I_t^s}(\omega_{t-1}^s)-V_{t-1}^s\big) \Big)\cr
    &&+(1-\phi_t^s-\tilde{\psi}_t^s)\text{grad}f_{I_t^s}(\omega_t^s) \vert \mathcal{F} _t^s]\cr
    &=&\phi_t^s\mathbb{E}[\text{grad}f_{I_t^s}(\omega_t^s)-\mathcal{T}_{\omega_0^s}^{\omega_t^s} \big(\text{grad}f_{I_t^s}(\omega_0^s)-\text{grad}f(\omega_0^s)\big) \vert \mathcal{F} _t^s]\cr
    &&+\tilde{\psi}_t^s\mathbb{E}[\text{grad}f_{I_t^s}(\omega_t^s)-\mathcal{T}_{\omega_{t-1}^s}^{\omega_t^s} \big(\text{grad}f_{I_t^s}(\omega_{t-1}^s)-V_{t-1}^s\big) \vert \mathcal{F} _t^s]\cr
    &&+(1-\phi_t^s-\tilde{\psi}_t^s)\mathbb{E}[\text{grad}f_{I_t^s}(\omega_t^s) \vert \mathcal{F} _t^s]\cr
    &=&\phi_t^s\mathbb{E}[\text{grad}f_{I_t^s}(\omega_t^s) \vert \mathcal{F} _t^s]-\mathcal{T}_{\omega_0^s}^{\omega_t^s} \big(\mathbb{E}[\text{grad}f_{I_t^s}(\omega_0^s \vert \mathcal{F} _t^s])-\text{grad}f(\omega_0^s)\big) \vert \mathcal{F} _t^s]\cr
    &&+\tilde{\psi}_t^s\mathbb{E}[\text{grad}f_{I_t^s}(\omega_t^s) \vert \mathcal{F} _t^s]-\mathcal{T}_{\omega_{t-1}^s}^{\omega_t^s} \big(\mathbb{E}[\text{grad}f_{I_t^s}(\omega_{t-1}^s) \vert \mathcal{F} _t^s]-V_{t-1}^s\big) \vert \mathcal{F} _t^s]\cr
    &&+(1-\phi_t^s-\tilde{\psi}_t^s)\mathbb{E}[\text{grad}f_{I_t^s}(\omega_t^s) \vert \mathcal{F} _t^s]\cr
    &=&\text{grad} f(\omega_t^s)+\tilde{\psi}_t^s\mathcal{T}_{\omega_{t-1}^s}^{\omega_t^s} \big(V_{t-1}^s-\text{grad}f(\omega_{t-1}^s)\big)
  \end{eqnarray}
\end{proof}

Now, we give the mean-square convergence of the proposed algorithm under the assumptions.
\begin{theorem}\label{T1}
  Suppose the assumption 1.b and assumption 2-4 hold. The sequences $\{\omega_t^s\}$ produced by algorithm 1, then $\{\mathbb{E}[\Vert\text{grad}f(\omega_t^s)\Vert^2]\}\rightarrow 0$
\end{theorem}

To prove theorem $\mathrm{\ref{T1}}$, we need the following two lemmas. These two lemmas are very useful in stochastic algorithms.

\begin{lemma}[\cite{RobbinsH1}]\label{L2}
    Let $\{x(k),\mathcal{F}(k)\}$, $\{\alpha(k),\mathcal{F}(k)\}$, $\{\beta(k),\mathcal{F}(k)\}$ and $\{\gamma(k),\mathcal{F}(k)\}$ be nonnegative adaptive
    sequences satisfying
    \begin{equation*}
           \mathbb{E}[x(k+1)\vert \mathcal{F}(k)] \leq (1+\alpha(k))x(k)-\beta(k)+\gamma(k),k\geq a.s.
    \end{equation*}
    if $\sum\limits_{k=0}^{\infty}(\alpha(k)+\gamma(k))<\infty$ a.s, then $x(k)$ converges to a
    finite random variable a.s. and $\sum\limits_{k=0}^{\infty}\beta(k)<\infty$ a.s.
\end{lemma}

\begin{lemma}[\cite{Fisk}]\label{L3}
    Let $\{x(k)\}$ be a nonnegative stochastic process with bounded positive variations,
    i.e., $\sum\limits_{k=0}^{\infty}\mathbb{E}[\mathbb{E}[x(k+1)-x(k)\vert \mathcal{F}(k)]^+]<\infty$,
    where $x^+=\max\{0,x\}$,
    Then $x(k)$ is a quasi martingale, i.e.,
    \begin{equation*}
        \sum\limits_{k=0}^{\infty}\vert \mathbb{E}[x(k+1)-x(k)\vert \mathcal{F}(k)]\vert<\infty \,a.s.\qquad x(k)\text{converges }\,a.s
    \end{equation*}
\end{lemma}

\noindent
\textbf{Proof of Theorem $\mathrm{\ref{T1}}$}
\begin{proof}
    Since $\Omega$ is compact, all continuous functions on $\Omega$ can be bounded.
    Hence, there exists a positive constant $N$, for all $\omega\in\Omega$, such that $\Vert\text{grad}f(\omega)\Vert\leq N$ and $\Vert\text{grad}f_i(\omega)\Vert\leq N,i=1,2,...,n$.
    We reindex the sequence $\{\alpha_t^s\},\{\omega_t^s\}$ as $\{x_0^1,x_1^1,..,x_{m-1}^1,x_0^2,x_1^2,...x_{m-1}^2,...\}$.
    From assumption 3, there exists $s_0$ such that for $s\geq s_0$, we have $\alpha_t^s\leq 1$, for any $t\in\{0,1,...,m-1\}$.
    Next we will use the mathematical induction proof that $s\geq s_0$, $\Vert V_t^s\Vert\leq 3N$.\\
    If $t=0$, $\Vert V_0^s\Vert=\Vert\text{grad}f(\omega_0^s)\Vert\leq N \leq 3N$, the conclusion is hold.\\
    Suppose that the conclusion is hold for any $t-1$, according to $\mathrm{(\ref{Alg1.1})}$ and assumption 4, we get
    \begin{eqnarray}\label{T1.P.1}
            \Vert V_t^s \Vert &\leq& \Vert\phi_t^s  \big(\text{grad}f_{I_t^s}(\omega_t^s)-\mathcal{T}_{\omega_0^s}^{\omega_t^s} \big(\text{grad}f_{I_t^s}(\omega_0^s)-\text{grad}f(\omega_0^s)\big) \big)\Vert\cr
              &&+\Vert\tilde{\psi}_t^s  \big(\text{grad}f_{I_t^s}(\omega_t^s)-\mathcal{T}_{\omega_{t-1}^s}^{\omega_t^s} \big(\text{grad}f_{I_t^s}(\omega_{t-1}^s)-V_{t-1}^s\big) \big)\Vert\cr
              &&+\Vert(1-\phi_t^s-\tilde{\psi}_t^s)  \text{grad}f_{I_t^s}(\omega_t^s)\Vert\cr
              &\leq& \phi_t^s  \big(\Vert\text{grad}f_{I_t^s}(\omega_t^s)\Vert+\Vert\big(\text{grad}f_{I_t^s}(\omega_0^s)\Vert+\Vert\text{grad}f(\omega_0^s)\Vert\big)\cr
              &&+\tilde{\psi}_t^s \big(\Vert\text{grad}f_{I_t^s}(\omega_t^s)\Vert+\Vert\text{grad}f_{I_t^s}(\omega_{t-1}^s)\Vert+\Vert V_{t-1}^s\Vert \big)\cr
              &&+(1-\phi_t^s-\tilde{\psi}_t^s) \Vert \text{grad}f_{I_t^s}(\omega_t^s)\Vert\cr
              &\leq& 3N
    \end{eqnarray}
    Then $s\geq s_0$, we have $\Vert V_t^s\Vert\leq 3N$. Denote $\Omega'=[0,1] \times\{(\omega,V)\vert \omega\in \Omega,V\in\mathbb{B}_{\omega}(0,3N)\}$.
  Defining $h(\tau,\omega,V):\Omega'\rightarrow\mathcal{R}$, $h(\tau,\omega,V):=(f\circ R_w)(-\tau V)$
  , from assumption 1.b and $\Omega'$ is a compact, there exists a constant $N'>0$
  such that $\vert \frac{\partial^2}{\partial \tau^2}h(\tau,\omega,V\vert)\leq N'$. By the Taylor expansion
  \begin{eqnarray}\label{T1.P.2}
    &f(\omega_{t+1}^s)-f(\omega_{t}^s)&=f\circ R_{\omega_{t}^s}(-\alpha_t^sV_t^s))-f\circ R_{\omega_{t}^s}(0)\cr
    &&=h(1,\omega_t^s,\alpha_t^sV_t^s)-h(0,\omega_t^s,vV_t^s)\cr
    &&=\frac{\partial}{\partial \tau }h(\tau,\omega_t^s,\alpha_t^sV_t^s)\vert _{\tau=0}\alpha_t^s+\int_0^1(1-\tau)\frac{\partial^2}{\partial \tau^2}h(\tau,\omega_t^s,\alpha_t^sV_t^s) d\tau\cr
    &&=\frac{\partial}{\partial \tau }h(\tau,\omega_t^s,\alpha_t^sV_t^s)\vert _{\tau=0}\alpha_t^s+\int_0^1(1-\tau)\frac{\partial^2}{\partial \tau^2}h(\alpha_t^s\tau,\omega_t^s,V_t^s) d\tau\cr
    &&\leq -\alpha_t^s\langle V_t^s,\text{grad}f(\omega_t^s)\rangle+\frac{N'}{2}(\alpha_t^s)^2
  \end{eqnarray}
  Since $\omega_t^s$ is measurable in $\mathcal{F}_t^s$, which together with lemma $\mathrm{\ref{L0}}$ and lemma $\mathrm{\ref{L1}}$ lead to
  \begin{eqnarray}\label{T1.P.3}
    &&\mathbb{E}[\langle V_t^s,\text{grad}f(\omega_t^s)\rangle\vert \mathcal{F}_t^s]\cr
    &=&\Vert \text{grad}f(\omega_t^s)\Vert^2+\tilde{\psi}_t^s\langle \mathcal{T}_{\omega_{t-1}^s}^{\omega_t^s} \big(V_{t-1}^s-\text{grad}f(\omega_{t-1}^s)\big),\text{grad}f(\omega_t^s)\rangle\cr
    &\geq& (1-\mu)\Vert \text{grad}f(\omega_t^s)\Vert^2
  \end{eqnarray}
  and
  \begin{eqnarray}\label{T1.P.4}
        &&\mathbb{E}[\langle V_t^s,\text{grad}f(\omega_t^s)\rangle\vert \mathcal{F}_t^s]\cr
        &=&\Vert \text{grad}f(\omega_t^s)\Vert^2+\tilde{\psi}_t^s\langle \mathcal{T}_{\omega_{t-1}^s}^{\omega_t^s} \big(V_{t-1}^s-\text{grad}f(\omega_{t-1}^s)\big),\text{grad}f(\omega_t^s)\rangle\cr
        &\leq& (1+\mu)\Vert \text{grad}f(\omega_t^s)\Vert^2
  \end{eqnarray}
  Then taking mathematical expectations on both sides of $\mathrm{(\ref{T1.P.2})}$, and substituting $\mathrm{(\ref{T1.P.3})}$ back to $\mathrm{(\ref{T1.P.2})}$ gives
  \begin{equation}\label{T1.P.5}
      \mathbb{E}[f(\omega_{t+1}^s)\vert \mathcal{F}_t^s]\leq f(\omega_{t}^s)-(1-\mu)\alpha_t^s\Vert \text{grad}f(\omega_t^s)\Vert^2+\frac{N'}{2}(\alpha_t^s)^2
  \end{equation}
  which yields
  \begin{equation}\label{T1.P.6}
      (1-\mu)\alpha_t^s \mathbb{E}[\Vert \text{grad}f(\omega_t^s)\Vert^2]\leq \mathbb{E}[f(\omega_{t}^s)-f(\omega_{t+1}^s)]+\frac{N'}{2}(\alpha_t^s)^2
  \end{equation}
  Summing this result over $t = 0,...,m-1$ and $s =1,...,S$ gives
  \begin{equation}\label{T1.P.7}
      \sum\limits_{s=s_0}^{S}\sum\limits_{t=0}^{m-1}(1-\mu)\alpha_t^s \mathbb{E}[\Vert \text{grad}f(\omega_t^s)\Vert^2]\leq \mathbb{E}[f(\omega_{0}^{s_0})-f(\omega_{m}^S)]+\sum\limits_{s=s_0}^{S}\sum\limits_{t=0}^{m-1}\frac{N'}{2}(\alpha_t^s)^2
  \end{equation}
  Since $f$ is continuous on $\Omega$, there exists a constant $C\geq 0$, for any $\omega \in \Omega$, $-C\leq f(\omega) \leq C$, which gives $f(\omega_{0}^{s_0})-f(\omega_{m}^S)\leq 2C$.
 Let $S\rightarrow \infty$, the above inequality gives
  \begin{equation}\label{T1.P.8}
      \sum\limits_{s=s_0}^{\infty}\sum\limits_{t=0}^{m-1}(1-\mu)\alpha_t^s \mathbb{E}[\Vert \text{grad}f(\omega_t^s)\Vert^2]\leq 2C+\sum\limits_{s=s_0}^{\infty}\sum\limits_{t=0}^{m-1}\frac{N'}{2}(\alpha_t^s)^2
  \end{equation}
  From assumption 3, $2C+\sum\limits_{s=s_0}^{\infty}\sum\limits_{t=0}^{m-1}\frac{N'}{2}(\alpha_t^s)^2=2C+\frac{N'}{2}\sum\limits_{s=s_0}^{\infty}\sum\limits_{t=0}^{m-1}(\alpha_t^s)^2<\infty$, such that
  \begin{equation}\label{T1.P.9}
      \sum\limits_{s=s_0}^{\infty}\sum\limits_{t=0}^{m-1}(1-\mu)\alpha_t^s\mathbb{E}[\Vert \text{grad}f(\omega_t^s)\Vert^2]<\infty
  \end{equation}
  Since $0<\mu<1$, we get $\sum\limits_{s=s_0}^{\infty}\sum\limits_{t=0}^{m-1}\alpha_t^s\mathbb{E}[\Vert \text{grad}f(\omega_t^s)\Vert^2]<\infty$, which together with
  $\sum\limits_{s=s_0}^{\infty}\sum\limits_{t=0}^{m-1}\alpha_t^s=\infty$ implies that
  $\lim\inf_{s\rightarrow\infty}\mathbb{E}[\Vert \text{grad}f(\omega_t^s)\Vert^2]=0$. Next, we will prove $\mathbb{E}[\Vert \text{grad}f(\omega_t^s)\Vert^2]$ converge to a real number.
  By assumption 3 and the properties of continuous functions on compact sets. Bounding the
  largest eigenvalue of the Hessian of $\Vert \text{grad}f(\omega_t^s)\Vert^2$ from above by $\gamma_1$
  along the curve defined by the retraction $R$ linking $\omega_{t+1}^s$ and $\omega_{t}^s$.
  A lower bound of the minimum eigenvalue of the Hessian of $f$ is $\gamma_2$.
  Let $\lambda_t^s$ be the eigenvalue of the Hessian of $f$ about $V_t^s$. By Taylor
  expansion, combining the above inequalities $\mathrm{(\ref{T1.P.3})}$ and $\mathrm{(\ref{T1.P.4})}$, we have
  \begin{eqnarray}\label{T1.P.10}
    &&\mathbb{E}[\Vert \text{grad}f(\omega_{t+1}^s)\Vert^2-\Vert \text{grad}f(\omega_t^s)\Vert^2\vert \mathcal{F}_t^s]\cr
    &\leq&\mathbb{E}[-2\alpha_t^s\langle \text{grad}f(\omega_t^s),\text{Hess}f(\omega_t^s)[V_t^s]\rangle+(\alpha_t^s)^2\Vert V_t^s\Vert^2\gamma_1\vert \mathcal{F}_t^s]\cr
    &=&-2\alpha_t^s\lambda_t^s \mathbb{E}[\langle \text{grad}f(\omega_t^s),V_t^s\rangle\vert \mathcal{F}_t^s]+(\alpha_t^s)^2\Vert V_t^s\Vert^2\gamma_1\cr
    &\leq&-2\alpha_t^s\gamma_2 \mathbb{E}[\langle \text{grad}f(\omega_t^s),V_t^s\rangle\vert \mathcal{F}_t^s]+(\alpha_t^s)^2\Vert V_t^s\Vert^2\gamma_1\cr
    &\leq&2\alpha_t^s \vert \gamma_2\vert(1+\mu)\Vert \text{grad}f(\omega_t^s)\Vert^2+(\alpha_t^s)^2(3N))^2\gamma_1\cr
    &=& 2\alpha_t^s \vert \gamma_2\vert(1+\mu)\Vert \text{grad}f(\omega_t^s)\Vert^2+(\alpha_t^s)^2(3N))^2\gamma_1
  \end{eqnarray}
  Taking mathematical expectations on both sides of $\mathrm{(\ref{T1.P.10})}$ yields
  \begin{eqnarray}\label{T1.P.11}
    &&\mathbb{E}[\Vert \text{grad}f(\omega_{t+1}^s)\Vert^2]\cr
    &\leq &\mathbb{E}[\Vert \text{grad}f(\omega_{t}^s)\Vert^2]+ 2\alpha_t^s \vert \gamma_2\vert(1+\mu)\mathbb{E}[\Vert \text{grad}f(\omega_t^s)]\Vert^2+(\alpha_t^s)^2(3N))^2\gamma_1
  \end{eqnarray}
  From assumption 3 and $\mathrm{(\ref{T1.P.9})}$, we know $\sum\limits_{s=s_0}^{\infty}\sum\limits_{t=0}^{m-1}2\alpha_t^s \vert \gamma_2\vert(1+\mu)\mathbb{E}[\Vert \text{grad}f(\omega_t^s)]\Vert^2+(\alpha_t^s)^2(3N))^2\gamma_1<\infty$.
  Using lemma $\mathrm{\ref{L2}}$(as $\mathbb{E}[\Vert \text{grad}f(\omega_{t+1}^s)\Vert^2]=x_t^s,\gamma_t^s2\alpha_t^s \vert \gamma_2\vert(1+\mu)\mathbb{E}[\Vert \text{grad}f(\omega_t^s)]\Vert^2+(\alpha_t^s)^2(3N))^2\gamma_1$), we get $\Vert \text{grad}f(\omega_t^s)\Vert^2$ convergence to a real number.
  Which together with $\lim\inf_{s\rightarrow\infty}\mathbb{E}[\Vert \text{grad}f(\omega_t^s)\Vert^2]=0$, then we have $\lim\limits_{s\rightarrow\infty}\mathbb{E}[\Vert \text{grad}f(\omega_t^s)\Vert^2]=0$.
  \end{proof}

Theorem $\mathrm{\ref{T1}}$ shows the convergence of the sequence produced by algorithm 1. According to theorem $\mathrm{\ref{T1}}$, we can obtain the convergence of the output of algorithm 1. Before this, we need to prove a lemma.
\begin{lemma}\label{L4}
    Let $a(k)$ be a real number sequence and satisfies $\lim\limits_{k\rightarrow \infty}a(k)=a$, then $\lim\limits_{k\rightarrow \infty}\frac{a(1)+a(2)+\cdots+a(k)}{k}=a$
\end{lemma}

\begin{proof}
    For any $\varepsilon >0$, there exists a positive constant $N_1$, such that $\vert a(k)-a\vert<\varepsilon$, $k> N_1$.
    Denote $M=\max\{\vert a(1)-a\vert,\cdots,\vert a(N_1)-a\vert\}$, we get
    \begin{eqnarray*}
        &&\Big\vert\frac{a(1)+a(2)+\cdots+a(k)}{k}-a\Big\vert=\Big\vert\frac{a(1)-a+a(2)-a+\cdots+a(k)-a}{k}\Big\vert \cr
        &\leq &\frac{1}{k}\big(\vert a(1)-a\vert+\vert a(2)-a\vert+\cdots+\vert a(k)-a\vert \big)\cr
        &\leq &\frac{MN_1}{k}+\frac{k-N_1}{k}\varepsilon <\frac{MN_1}{k}+\varepsilon
    \end{eqnarray*}
    Note that $\lim\limits_{k\rightarrow \infty}\frac{MN_1}{k}=0$, for the above $\varepsilon$, there exists a constant $N_2$, such that $\frac{MN_1}{k}<\varepsilon$, $k>N_2$. 
    Denote $N=\max\{N_1,N_2\}$, $\vert\frac{a(1)+a(2)+\cdots+a(k)}{k}-a\vert<2\varepsilon$, $k>N$. 
  \end{proof}
\begin{corollary}
    Suppose the assumption 1.b and assumption 2-4 hold. The sequences $\{\omega_t^s\}$ produced by algorithm 1.
    No matter what  choose, we have $\{\mathbb{E}[\Vert\text{grad}f(\omega_a)\Vert^2]\}\rightarrow 0$
\end{corollary}

\begin{proof}
    If we choose option 1, $\lim\limits_{s\rightarrow\infty} \mathbb{E}[\Vert \text{grad}f(\omega_t^s)\Vert^2]=0$ implies that $\mathbb{E}[\Vert\text{grad}f(\omega_a)\Vert^2]\rightarrow 0$. 
    If we choose option 2, according to theorem $\mathrm{\ref{T1}}$,
    for all $t$, we have $\lim\limits_{S\rightarrow\infty} \mathbb{E}[\Vert\text{grad}f(\omega_t^S)\Vert^2]=0$, using lemma $\mathrm{\ref{L4}}$, then $\lim\limits_{S\rightarrow\infty}\mathbb{E}[\Vert\text{grad}f(\omega_a)\Vert^2]=\lim\limits_{S\rightarrow\infty}\frac{1}{mS}\sum\limits_{s=1}^{S}\sum\limits_{t=0}^{m-1}\mathbb{E}[\Vert\text{grad}f(\omega_t^s)\Vert^2]=0$\\
\end{proof}

The above theorem has no special requirements for function $f$.
  The following theorem introduces that a better conclusion can be obtained
   after the function $f$ satisfies other properties. That is,
  if the function $f$ satisfies $f\geq 0$, we can get an almost sure convergence.

\begin{theorem}\label{T2}
    Suppose the assumption 1.b and assumption 2-4 hold. The sequences $\{\omega_t^s\}$ produced by algorithm 1.
    If $f\geq 0$, then $\{\text{grad}f(\omega_t^s)\}\rightarrow0\,a.s.$ and
    $\{f(\omega_t^s)\}$ converges to a finite random variable $\,a.s.$.
\end{theorem}

\begin{remark}
    If $\psi_t^s\equiv 0$ is hold, we can be degenerated into the situation in \cite{SatoH1}.
\end{remark}

\begin{proof}
    From the inequality $\mathrm{(\ref{T1.P.5})}$, we have
    \begin{equation}\label{T2.P.1}
        \mathbb{E}[f(\omega_{t+1}^s)\vert \mathcal{F}_t^s]\leq f(\omega_{t}^s)-(1-\mu)\alpha_t^s\Vert \text{grad}f(\omega_t^s)\Vert^2+\frac{N'}{2}(\alpha_t^s)^2
    \end{equation}
    We can get from assumption 3, that is $\sum\limits_{s=s_0}^{\infty}\sum\limits_{t=0}^{m-1}\frac{N'}{2}(\alpha_t^s)^2=\frac{N'}{2}\sum\limits_{s=s_0}^{\infty}\sum\limits_{t=0}^{m-1}(\alpha_t^s)^2<\infty$.
    By condition $f\geq 0$, using lemma $\mathrm{\ref{L2}}$, then $\{f(\omega_t^s)\}$ converges to a finite random variable $\,a.s$, and
    \begin{equation}\label{T2.P.2}
        \sum\limits_{s=s_0}^{\infty}\sum\limits_{t=0}^{m-1}(1-\mu)\alpha_t^s\Vert \text{grad}f(\omega_t^s)\Vert^2<\infty\,a.s.
    \end{equation}
    Moreover, since $0<\mu<1$, we obtain
    \begin{equation}\label{T2.P.3}
        \sum\limits_{s=s_0}^{\infty}\sum\limits_{t=0}^{m-1}\alpha_t^s\Vert \text{grad}f(\omega_t^s)\Vert^2<\infty\,a.s
    \end{equation}
    which gives together with $\sum\limits_{s=s_0}^{\infty}\sum\limits_{t=0}^{m-1}\alpha_t^s=\infty$ leads to
    $\lim\inf_{s\rightarrow\infty}\Vert \text{grad}f(\omega_t^s)\Vert^2=0\,a.s$. Next, we will prove that $\Vert \text{grad}f(\omega_t^s)\Vert^2$ converges to a finite random variable $\,a.s$.
    According to the inequalities $\mathrm{(\ref{T1.P.10})}$
    \begin{equation}\label{T2.P.4}
        \begin{split}
        &\mathbb{E}[\Vert \text{grad}f(\omega_{t+1}^s)\Vert^2-\Vert \text{grad}f(\omega_t^s)\Vert^2\vert \mathcal{F}_t^s]^+\cr
        =&\max \{0,\mathbb{E}[\Vert \text{grad}f(\omega_{t+1}^s)\Vert^2-\Vert \text{grad}f(\omega_t^s)\Vert^2\vert \mathcal{F}_t^s]\}\cr
        \leq&\max \{0, 2\alpha_t^s \vert \gamma_2\vert(1+\mu)\Vert \text{grad}f(\omega_t^s)\Vert^2+(\alpha_t^s)^2(3N))^2\gamma_1\}\cr
         =& 2\alpha_t^s \vert \gamma_2\vert(1+\mu)\Vert \text{grad}f(\omega_t^s)\Vert^2+(\alpha_t^s)^2(3N))^2\gamma_1
        \end{split}
    \end{equation}
    this together with assumption 3 and $\mathrm{(\ref{T2.P.3})}$, we get $\sum\limits_{s=s_0}^{\infty}\sum\limits_{t=0}^{m-1}\big(2\alpha_t^s \vert \gamma_2\vert(1+\mu)\Vert \text{grad}f(\omega_t^s)\Vert^2+(\alpha_t^s)^2(3N))^2\gamma_1 \big)<\infty$.
    From lemma $\mathrm{\ref{L3}}$, $\Vert \text{grad}f(\omega_t^s)\Vert^2$ is a quasi martingale, i.e., $\Vert \text{grad}f(\omega_t^s)\Vert^2$ converges to a finite random variable $\,a.s$.
    Combining $\lim\inf_{s\rightarrow\infty}\Vert \text{grad}f(\omega_t^s)\Vert^2=0\,a.s.$ gives the desired result $\lim\limits_{s\rightarrow\infty}\Vert \text{grad}f(\omega_t^s)\Vert^2=0\,a.s.$
\end{proof}

\subsection{Step size is fixed}
    Theorem $\mathrm{\ref{T1}}$ and theorem $\mathrm{\ref{T2}}$ qualitatively research the convergence of R-SHG with adaptive parameters when the step size is reduced.
    For the case of fixed step size, theorem $\mathrm{\ref{T1}}$ and theorem $\mathrm{\ref{T2}}$ will not be satisfied, and the reason is that
    the conditions of assumption 3 will not be satisfied. Furthermore, we can use the weaker
    differentiability of the function of $f$, and quantitatively research
    the convergence speed of the algorithm. Before that, we will prove the following two lemmas.
    At the rest of this article, we suppose that the $N$ is defined by theorem $\mathrm{\ref{T1}}$, i.e., for any $\omega\in \Omega$, $\Vert \text{grad}f_i(\omega)\Vert\leq N$.
\begin{lemma}\label{L5}
    Suppose assumption 1.a, assumption 5 and assumption 6 hold, for any $\omega_1,\omega_2\in \Omega$ are $\mathcal{F}_t^s$ measurable, $\omega_2=R_{\omega_1}(\xi_{\omega_1}^{\omega_2})$, such that
    \begin{equation}\label{L5.1}
        \mathbb{E}[\Vert \text{grad}f(\omega_2)-\mathcal{T}_{\omega_1}^{\omega_2} \text{grad}f(\omega_1)\Vert^2\vert \mathcal{F}_t^s]\leq 2(M^2+\theta^2N^2)\Vert \xi_{\omega_1}^{\omega_2}\Vert^2
    \end{equation}
\end{lemma}

\begin{lemma}\label{L6}
    Suppose assumption 1.a and assumption 4-7 hold, we have
    \begin{equation}
        \begin{split}\label{L6.1}
            \mathbb{E}[\Vert V_t^s-\text{grad}f(\omega_t^s)\Vert^2\vert \mathcal{F}_t^s]
           &\leq 4(M^2+\theta^2N^2)(\phi_t^s)^2\Vert \xi_{\omega_0^s}^{\omega_t^s}\Vert^2+4N^2\big((1-\phi_t^s)^2\cr
           &+(\psi_t^s)^2\big)+(\psi_t^s)^2\Vert V_{t-1}^s-\text{grad}f(\omega_{t-1}^s)\Vert^2
        \end{split}
    \end{equation}
\end{lemma}
The proofs of the above two lemmas are shown in Appendix $\mathrm{\ref{secA}}$. Lemma $\mathrm{\ref{L6}}$ gives the mean-square bound between $V_t^s$ and $\text{grad}f(\omega_t^s)$. The idea of lemma $\mathrm{\ref{L6}}$ is important and there are many similar proofs in this paper.
Now, we will give the third theorem that research the convergence speed of algorithm 1
when step size and inner loop parameters are fixed. But the convergence of the algorithm is local.

\begin{theorem}\label{T3}
    Suppose assumption 1.a, assumption 2 and assumption 4-9 hold. The sequences $\{\omega_t^s\}$ produced by algorithm 1 with option 2.
      Let $\alpha_t^s\equiv C_\alpha$, $\phi_t^s=\phi^s$ and $\psi_t^s=\psi^s$, satisfying
      $C_\alpha\leq\frac{2}{L+\sqrt{L^2+4\nu}}$.
      where $\nu=4m^3(M^2+\theta^2N^2)C_1^2C_2^2$,
      and $\sum\limits_{s=1}^{\infty}((1-\phi^s)^2+(\psi^s)^2<\infty$, then
      \begin{eqnarray}\label{T3.1}
        &\mathbb{E}[\Vert \text{grad}f(\omega_{a})\Vert^2]&\leq\frac{2}{mSC_\alpha}(f(\tilde{\omega}^0)-f(\omega^*))+\frac{4mN^2}{S}\sum\limits_{s=1}^{\infty}((1-\phi^s)^2+(\psi^s)^2)\cr
        &&=\mathcal{O}(\frac{1}{S})
      \end{eqnarray}
  \end{theorem}

  \begin{remark}
    It is easily verified that $\phi^s$ and $\psi^s$ satisfy the condition of theorem $\mathrm{\ref{T3}}$,
      if $\phi^s=1-\frac{1}{s+1},\psi^s=\frac{1}{(s+1)^2}$. It is also hold $\phi^s\equiv 1$, i.e., $\psi^s\equiv 0$.
      To compare with \cite{SatoH1}, we also give the convergence rate analysis under fixed step size.
  \end{remark}

\begin{proof}
    Using assumption 8, we get
    \begin{eqnarray}\label{T3.P.1}
        &&f(\omega_{t+1}^s)-f(\omega_{t}^s)\cr
        &\leq& -C_\alpha \langle\text{grad}f(\omega_{t}^s),V_t^s\rangle+\frac{LC_\alpha^2}{2}\Vert V_t^s\Vert^2\cr
        &=&-\frac{C_\alpha}{2}\Vert V_t^s\Vert^2-\frac{C_\alpha}{2}\Vert \text{grad}f(\omega_{t}^s)\Vert^2+\frac{C_\alpha}{2}\Vert V_t^s-\text{grad}f(\omega_{t}^s)\Vert^2+\frac{LC_\alpha^2}{2}\Vert V_t^s\Vert^2\cr
        &=&-\frac{C_\alpha}{2}\Vert \text{grad}f(\omega_{t}^s)\Vert^2+\frac{C_\alpha}{2}\Vert V_t^s-\text{grad}f(\omega_{t}^s)\Vert^2+(\frac{LC_\alpha^2}{2}-\frac{C_\alpha}{2})\Vert V_t^s\Vert^2
    \end{eqnarray}
    Taking the mathematical expectations on both sides of $\mathrm{(\ref{T3.P.1})}$, thus
    \begin{eqnarray}\label{T3.P.2}
        &&\mathbb{E}[\Vert \text{grad}f(\omega_{t}^s)\Vert^2]\cr
        &\leq &\frac{2}{C_\alpha}\mathbb{E}[f(\omega_{t}^s)-f(\omega_{t+1}^s)]+\mathbb{E}[\Vert V_t^s-\text{grad}f(\omega_{t}^s)\Vert^2]+(LC_\alpha-1)\mathbb{E}[\Vert V_t^s\Vert^2]
    \end{eqnarray}
    According to assumption 9, we obtain
    \begin{eqnarray}\label{T3.P.3}
        &&\Vert \xi_{\omega_0^s}^{\omega_t^s}\Vert^2=\Vert R_{\omega_0^s}^{-1}(\omega_t^s)\Vert^2\cr
        &\leq& C_1^2d^2(\omega_0^s,\omega_t^s)\cr
        &\leq& C_1^2(d(\omega_0^s,\omega_1^s)+d(\omega_1^s,\omega_2^s)+...+d(\omega_{t-1}^s,\omega_t^s))^2\cr
        &\leq& C_1^2t(d^2(\omega_0^s,\omega_1^s)+d^2(\omega_1^s,\omega_2^s)+...+d^2(\omega_{t-1}^s,\omega_t^s))\cr
        &\leq& C_1^2C_2^2C_\alpha^2t\sum\limits_{i=0}^{t-1}\Vert V_i^s\Vert^2
    \end{eqnarray}
    Let the parameters, step size and $\mathrm{(\ref{T3.P.3})}$ back to $\mathrm{(\ref{L6.1})}$, we have
    \begin{eqnarray}\label{T3.P.4}
        &&\mathbb{E}[\Vert V_t^s-\text{grad}f(\omega_t^s)\Vert^2\vert \mathcal{F}_t^s]\cr
        &\leq&4(M^2+\theta^2N^2)(\phi^s)^2C_1^2C_2^2C_\alpha^2t\sum\limits_{i=0}^{t-1}\Vert V_i^s\Vert^2+4N^2((1-\phi^s)^2+(\psi^s)^2)\cr
        &&+(\psi^s)^2\Vert V_{t-1}^s-\text{grad}f(\omega_{t-1}^s)\Vert^2\cr
        &\leq&4(M^2+\theta^2N^2)C_1^2C_2^2C_\alpha^2m\sum\limits_{t=0}^{m-1}\Vert V_t^s\Vert^2+4N^2((1-\phi^s)^2+(\psi^s)^2)\cr
        &&+\Vert V_{t-1}^s-\text{grad}f(\omega_{t-1}^s)\Vert^2\cr
        &\leq&4(M^2+\theta^2N^2)C_1^2C_2^2C_\alpha^2mt\sum\limits_{t=0}^{m-1}\Vert V_t^s\Vert^2+4tN^2((1-\phi^s)^2+(\psi^s)^2)\cr
        &&+\Vert V_{0}^s-\text{grad}f(\omega_{0}^s)\Vert^2\cr
        &\leq&4(M^2+\theta^2N^2)C_1^2C_2^2C_\alpha^2m^2\sum\limits_{t=0}^{m-1}\Vert V_t^s\Vert^2+4mN^2((1-\phi^s)^2+(\psi^s)^2)
    \end{eqnarray}
    The second inequality is due to $0\leq \phi^s\leq 1, 0\leq \psi^s\leq 1 $ and $t\leq m$.
    The last inequality is bases on the fact that $V_{0}^s=\text{grad}f(\omega_{0}^s)$.
    Note that $\mathbb{E}[\mathbb{E}[\cdot\vert \mathcal{F}_t^s]]=\mathbb{E}[\cdot]$,
    taking the mathematical expectations with respect to $\mathrm{(\ref{T3.P.4})}$, we get
    \begin{eqnarray}\label{T3.P.5}
        &&\sum\limits_{s=1}^S\sum\limits_{t=0}^{m-1}\mathbb{E}[\Vert V_t^s-\text{grad}f(\omega_t^s)\Vert^2]\cr
        &\leq&\sum\limits_{s=1}^S\sum\limits_{t=0}^{m-1}\Big(4(M^2+\theta^2N^2)C_1^2C_2^2C_\alpha^2m^2\sum\limits_{t=0}^{m-1}\mathbb{E}[\Vert V_t^s\Vert^2]+4mN^2((1-\phi^s)^2+(\psi^s)^2)\Big)\cr
        &=&\sum\limits_{s=1}^S\Big(4(M^2+\theta^2N^2)C_1^2C_2^2C_\alpha^2m^3\sum\limits_{t=0}^{m-1}\mathbb{E}[\Vert V_t^s\Vert^2]+4m^2N^2((1-\phi^s)^2+(\psi^s)^2)\Big)\cr
        &=&4(M^2+\theta^2N^2)C_1^2C_2^2C_\alpha^2m^3\sum\limits_{s=1}^S\sum\limits_{t=0}^{m-1}\mathbb{E}[\Vert V_t^s\Vert^2]+4m^2N^2\sum\limits_{s=1}^S((1-\phi^s)^2+(\psi^s)^2)\cr
        &\leq &4(M^2+\theta^2N^2)C_1^2C_2^2C_\alpha^2m^3\sum\limits_{s=1}^S\sum\limits_{t=0}^{m-1}\mathbb{E}[\Vert V_t^s\Vert^2]+4m^2N^2\sum\limits_{s=1}^{\infty}((1-\phi^s)^2+(\psi^s)^2)
    \end{eqnarray}
    Summing the result of $\mathrm{(\ref{T3.P.2})}$ gives
    \begin{eqnarray}\label{T3.P.6}
        &&\sum\limits_{s=1}^S\sum\limits_{t=0}^{m-1}\mathbb{E}[\Vert \text{grad}f(\omega_{t}^s)\Vert^2] \cr
        &\leq& \sum\limits_{s=1}^S\sum\limits_{t=0}^{m-1}\mathbb{E}[\Vert V_t^s-\text{grad}f(\omega_{t}^s)\Vert^2]+(LC_\alpha-1)\sum\limits_{s=1}^S\sum\limits_{t=0}^{m-1}\mathbb{E}[\Vert V_t^s\Vert^2]\cr
        &&+\frac{2}{C_\alpha}\mathbb{E}[f(\tilde{\omega}^0)-f(\omega_m^S)]\cr
        &\leq&\Big(4m^3(M^2+\theta^2N^2)C_1^2C_2^2C_\alpha^2+LC_\alpha-1)\Big)\sum\limits_{s=1}^S\sum\limits_{t=0}^{m-1} \mathbb{E}[\Vert V_t^s\Vert^2]\cr
        &&+4m^2N^2\sum\limits_{s=1}^{\infty}((1-\phi^s)^2+(\psi^s)^2)+\frac{2}{C_\alpha}\mathbb{E}[f(\tilde{\omega}^0)-f(\omega_m^S)]\cr
        &\leq&\frac{2}{C_\alpha}\mathbb{E}[f(\tilde{\omega}^0)-f(\omega_m^S)]+4m^2N^2\sum\limits_{s=1}^{\infty}((1-\phi^s)^2+(\psi^s)^2)\cr
        &\leq&\frac{2}{C_\alpha}(f(\tilde{\omega}^0)-f(\omega^*))+4m^2N^2\sum\limits_{s=1}^{\infty}((1-\phi^s)^2+(\psi^s)^2)
    \end{eqnarray}
    The third inequality applies $4m^3(M^2+\theta^2N^2)C_1^2C_2^2C_\alpha^2+LC_\alpha-1\leq 0$, if $C_\alpha\leq\frac{2}{L+\sqrt{L^2+4\nu}}$.
    The last inequality follows from $f(\omega_t^S)>f(\omega^*)$. Hence, we have
    \begin{eqnarray}\label{T3.P.7}
        &\mathbb{E}[\Vert \text{grad}f(\omega_{a})\Vert^2]&=\frac{1}{mS}\sum\limits_{s=1}^S\sum\limits_{t=0}^{m-1}\mathbb{E}[\Vert \text{grad}f(\omega_{t}^s)\Vert^2] \cr
        &&\leq\frac{2}{mSC_\alpha}(f(\tilde{\omega}^0)-f(\omega^*))+\frac{4mN^2}{S}\sum\limits_{s=1}^{\infty}((1-\phi^s)^2+(\psi^s)^2)\cr
        &&=\mathcal{O}(\frac{1}{S})
    \end{eqnarray}
\end{proof}

\subsection{Special case 1}
Suppose the descent direction only use R-SVRG and R-SRG term, i.e.,
\begin{equation}\label{Alg2}
    \begin{split}
        V_t^s&=\tilde{\phi}_t^s\Big(\text{grad}f_{I_t^s}(\omega_t^s)-\mathcal{T}_{\omega_0^s}^{\omega_t^s} \big(\text{grad}f_{I_t^s}(\omega_0^s)-\text{grad}f(\omega_0^s)\big) \Big)\cr
        &+\tilde{\psi}_t^s\Big(\text{grad}f_{I_t^s}(\omega_t^s)-\mathcal{T}_{\omega_{t-1}^s}^{\omega_t^s} \big(\text{grad}f_{I_t^s}(\omega_{t-1}^s)-V_{t-1}^s\big) \Big)\cr
    \end{split}
\end{equation}
where $\tilde{\phi}_t^s=1-\tilde{\psi}_t^s$.
For the step size is reduced, lemma $\mathrm{\ref{L0}}$ and lemma $\mathrm{\ref{L1}}$ are satisfied. Therefore, theorem $\mathrm{\ref{T1}}$ and theorem $\mathrm{\ref{T2}}$
are still hold. Compared to theorem $\mathrm{\ref{T3}}$, we can get a similar conclusion under weaker
conditions; that is, it is not necessary to fix the inner loop
parameters. Before this, let us give a lemma.

\begin{lemma}\label{L7}
    Suppose assumption 1.a and assumption 4-7 hold, let the descent direction be $\mathrm{(\ref{Alg2})}$, then
    \begin{equation}\label{L7.1}
        \begin{split}
            \mathbb{E}[\Vert V_t^s-\text{grad}f(\omega_t^s)\Vert^2\vert \mathcal{F}_t^s]&\leq 4(M^2+\theta^2N^2)(\tilde{\phi}_t^s)^2\Vert \xi_{\omega_0^s}^{\omega_t^s}\Vert^2+8N^2(\tilde{\psi}_t^s)^2\cr
            &+(\tilde{\psi}_t^s)^2\Vert V_{t-1}^s-\text{grad}f(\omega_{t-1}^s)\Vert^2
        \end{split}
    \end{equation}
\end{lemma}

The proof of the lemma is in Appendix $\mathrm{\ref{secA}}$.

\begin{theorem}\label{T4}
    Suppose assumption 1.a, assumption 2 and assumption 4-9 hold. The sequences $\{\omega_t^s\}$
    produced by algorithm 1 with option 2 and the descent direction is $\mathrm{(\ref{Alg2})}$. Let $\alpha_t^s\equiv C_\alpha$, 
    satisfying $C_\alpha\leq\frac{1}{L+\sqrt{L^2+4\nu}}$, 
    where $\nu=4m^3(M^2+\theta^2N^2)C_1^2C_2^2$, and 
    $\sum\limits_{s=1}^{\infty}\sum\limits_{t=0}^{m-1}(\psi_t^s)^2<\infty$, 
    such that
     \begin{equation}\label{T4.1}
         \begin{split}
             \mathbb{E}[\Vert \text{grad}f(\omega_{a})\Vert^2]&\leq\frac{2}{mSC_\alpha}\mathbb{E}[f(\tilde{\omega}^0)-f(\omega^*)]+\frac{8N^2}{S}\sum\limits_{s=1}^{\infty}\sum\limits_{t=0}^{m-1}(\psi_t^s)^2\cr
             &=\mathcal{O}(\frac{1}{S})
         \end{split}
     \end{equation}
 \end{theorem}

 \begin{proof}
     Let the parameters, step size and $\mathrm{(\ref{T3.P.3})}$ back to $\mathrm{(\ref{L7.1})}$ gives 
     \begin{eqnarray}\label{T4.P.1}
       &&\mathbb{E}[\Vert V_t^s-\text{grad}f(\omega_t^s)\Vert^2\vert \mathcal{F}_t^s]\cr
       &\leq&4(M^2+\theta^2N^2)(\tilde{\phi}_t^s)^2C_1^2C_2^2C_\alpha^2t\sum\limits_{i=0}^{t-1}\Vert V_i^s\Vert^2+8N^2(\tilde{\psi}_t^s)^2)\cr
       &&+(\tilde{\psi}_t^s)^2\Vert V_{t-1}^s-\text{grad}f(\omega_{t-1}^s)\Vert^2\cr
       &\leq&4(M^2+\theta^2N^2)C_1^2C_2^2C_\alpha^2m\sum\limits_{t=0}^{m-1}\Vert V_t^s\Vert^2+8N^2(\tilde{\psi}_t^s)^2)\cr
       &&+\Vert V_{t-1}^s-\text{grad}f(\omega_{t-1}^s)\Vert^2\cr
       &\leq&4(M^2+\theta^2N^2)C_1^2C_2^2C_\alpha^2m^2\sum\limits_{t=0}^{m-1}\Vert V_t^s\Vert^2+8mN^2(\tilde{\psi}_t^s)^2)\cr
       &&+\Vert V_{0}^s-\text{grad}f(\omega_{0}^s)\Vert^2\cr
       &=&4(M^2+\theta^2N^2)C_1^2C_2^2C_\alpha^2m^2\sum\limits_{t=0}^{m-1}\Vert V_t^s\Vert^2+8mN^2(\psi_t^s)^2)
     \end{eqnarray}
     The above inequality applies $0\leq \title{\phi}_t^s\leq 1, \title{\psi}_t^s \leq \psi_t^s$ and $t\leq m$.
     Similar to the proof of $\mathrm{(\ref{T3.P.1})}$-$\mathrm{(\ref{T3.P.6})}$ in theorem $\mathrm{\ref{T3}}$, we get
     \begin{eqnarray}\label{T4.P.2}
           &&\sum\limits_{s=1}^S\sum\limits_{t=0}^{m-1}\mathbb{E}[\Vert \text{grad}f(\omega_{t}^s)\Vert^2] \cr
           &\leq& \sum\limits_{s=1}^S\sum\limits_{t=0}^{m-1}\mathbb{E}[\Vert V_t^s-\text{grad}f(\omega_{t}^s)\Vert^2]+(LC_\alpha-1)\sum\limits_{s=1}^S\sum\limits_{t=0}^{m-1}\mathbb{E}[\Vert V_t^s\Vert^2]\cr
           &&+\frac{2}{C_\alpha}\mathbb{E}[f(\tilde{\omega}^0)-f(\omega_m^S)]\cr
           &\leq&\Big(4m^3(M^2+\theta^2N^2)C_1^2C_2^2C_\alpha^2+LC_\alpha-1)\Big)\sum\limits_{s=1}^S\sum\limits_{t=0}^{m-1} \mathbb{E}[\Vert V_t^s\Vert^2]\cr
           &&+8mN^2\sum\limits_{s=1}^{\infty}(\tilde{\psi}_t^s)^2+\frac{2}{C_\alpha}\mathbb{E}[f(\tilde{\omega}^0)-f(\omega_m^S)]\cr
           &\leq&\frac{2}{C_\alpha}\mathbb{E}[f(\tilde{\omega}^0)-f(\omega_m^S)]+8mN^2\sum\limits_{s=1}^{\infty}\sum\limits_{t=0}^{m-1}(\psi_t^s)^2\cr
           &\leq&\frac{2}{C_\alpha}(f(\tilde{\omega}^0)-f(\omega^*))+8mN^2\sum\limits_{s=1}^{\infty}\sum\limits_{t=0}^{m-1}(\psi_t^s)^2
     \end{eqnarray}
     Hence, we obtain
     \begin{eqnarray}\label{T4.P.3}
       &\mathbb{E}[\Vert \text{grad}f(\omega_{a})\Vert^2]&=\frac{1}{mS}\sum\limits_{s=1}^S\sum\limits_{t=0}^{m-1}\mathbb{E}[\Vert \text{grad}f(\omega_{t}^s)\Vert^2] \cr
       &&\leq\frac{2}{mSC_\alpha}(f(\tilde{\omega}^0)-f(\omega^*))+\frac{8N^2}{S}\sum\limits_{s=1}^{\infty}\sum\limits_{t=0}^{m-1}(\psi_t^s)^2\cr
             &&=\mathcal{O}(\frac{1}{S})
     \end{eqnarray}
 \end{proof}

 \subsection{Special case 2}
 The previous two subsections present a local convergence rate analysis of the algorithm with retraction mapping and vector transport. In this subsection, we consider a special case of the result in the previous subsection, where exponential mapping and parallel translation are chosen as retraction and vector transport. The previous theorems still hold when the retraction mapping is taken as exponential mapping and the vector transport is taken as parallel transport. For theorem $\mathrm{\ref{T3}}$, if the exponential mapping and parallel transport are used, then the convergence is global convergence. For this special case, we give only a sketch of the proofs and the result as the following corollary.
 \begin{corollary}\label{C2}
     Suppose the conditions in theorem $\mathrm{\ref{T3}}$ are hold and consider algorithm 1 with $R=Exp$ and $\mathcal{T}=\Gamma$. Let $C_\alpha\leq\frac{2}{L+\sqrt{L^2+4\nu}}$ and
     $\sum\limits_{s=1}^{\infty}((1-\phi^s)^2+(\psi^s)^2<\infty$
     where $\nu=4m^3M^2$, such that
     \begin{equation}\label{C2.1}
         \begin{split}
             \mathbb{E}[\Vert \text{grad}f(\omega_{a})\Vert^2]&\leq\frac{2}{mSC_\alpha}(f(\tilde{\omega}^0)-f(\omega^*))+\frac{4mN^2}{S}\sum\limits_{s=1}^{\infty}((1-\phi^s)^2+(\psi^s)^2)\cr
             &=\mathcal{O}(\frac{1}{S})
         \end{split}
     \end{equation}
 \end{corollary}

 \begin{proof}
   If the retraction mapping is taken as exponential mapping and the vector transport
   is taken as parallel transport, then the inequality $\mathrm{(\ref{L6.1})}$ in lemma $\mathrm{\ref{L6}}$ becoming
     \begin{equation}\label{C2.P.1}
         \begin{split}
             \mathbb{E}[\Vert V_t^s-\text{grad}f(\omega_t^s)\Vert^2\vert \mathcal{F}_t^s]
            &\leq 4M^2(\phi_t^s)^2\Vert \xi_{\omega_0^s}^{\omega_t^s}\Vert^2+4N^2\big((1-\phi_t^s)^2\cr
            &+(\psi_t^s)^2\big)+(\psi_t^s)^2\Vert V_{t-1}^s-\text{grad}f(\omega_{t-1}^s)\Vert^2
         \end{split}
     \end{equation}
     And the inequality $\mathrm{(\ref{T3.P.3})}$ in theorem $\mathrm{\ref{T3}}$ becoming
     \begin{equation}\label{C2.P.2}
         \begin{split}
             \Vert \xi_{\omega_0^s}^{\omega_t^s}\Vert^2\leq C_\alpha^2t\sum\limits_{i=0}^{t-1}\Vert V_i^s\Vert^2
         \end{split}
     \end{equation}
     Other proofs are the same as theorem $\mathrm{\ref{T3}}$
 \end{proof}
 \begin{remark}
   These are equivalent to $\theta=0,C_1=C_2=1$. But the convergence rate of corollary is global.
 \end{remark}

 \section{Riemannian Stochastic Hybrid Gradient Algorithm with time-varying Parameters}
 In this section, we will prove the convergence rate under both step size are reduced and fixed.
When the step size is reduced, if we choose $\omega_a=\tilde{\omega}^S$, it is difficult to analyse the convergence of $\omega_a$. It is different from algorithm 1, and we only consider option 2 of R-SHG algorithm with time-varying parameters to analyse the convergence. We can quantitatively research the convergence of $\omega_a$. Of course, the advantage is that we only need to use assumption 1.a. Here we will propose the second algorithm in this paper.\\
 \begin{breakablealgorithm}\label{Alg3}
     \footnotesize
     \caption{R-SHG algorithm with time-varying parameters}
     \begin{algorithmic}[1]
         \State\textbf {Input:}  step size $\alpha_t^s$, frequency $m>0$, the positive parameters $\psi_t^s$, $\phi_t^s$.
         \State\textbf {Initialize:}  $\tilde{\omega}^0$;
         \For {s=1,2,...,S }
         \State Caclulate the full Riemannian gradient $\text{grad}f(\tilde{\omega}^{s-1})$
         \State Store $\omega_0^s=\tilde{\omega}^{s-1}$, $V_0^s=\text{grad}f(\tilde{\omega}_0^s)$, $\omega_1^s=R_{\omega_0^s}(-\alpha_0^sV_0^s)$
         \For {t=1,2,...,m-1}
         \State Choose $I_t^s\in\{1,2,...,n\}$ uniformly at random
         \State Caclulate the descent direction
                 \begin{equation*}
                         \begin{split}
                             V_t^s&=\phi_t^s\Big(\text{grad}f_{I_t^s}(\omega_t^s)-\mathcal{T}_{\omega_0^s}^{\omega_t^s} \big(\text{grad}f_{I_t^s}(\omega_0^s)-\text{grad}f(\omega_0^s)\big) \Big)\cr
                             &+\psi_t^s\Big(\text{grad}f_{I_t^s}(\omega_t^s)-\mathcal{T}_{\omega_{t-1}^s}^{\omega_t^s} \big(\text{grad}f_{I_t^s}(\omega_{t-1}^s)-V_{t-1}^s\big) \Big)\cr
                             &+(1-\phi_t^s-\tilde{\psi}_t^s)\text{grad}f_{I_t^s}(\omega_t^s)
                         \end{split}
                     \end{equation*}
         \State$\omega_{t+1}^s=R_{\omega_t^s}(-\alpha_t^sV_t^s)$
         \EndFor
         \State $\tilde{\omega}^s=\omega_m^s$
         \EndFor
         \State We chosen ${\omega}_a$ uniformly randomly from $\{\{\omega_t^s\}_{t=0}^{m-1}\}_{s=1}^S$
         \State \textbf {Onput:} ${\omega}_a$
     \end{algorithmic}
 \end{breakablealgorithm}
 \begin{remark}
     If $\psi_t^s\equiv 0$, then the two algorithms are equivalent. Therefore, the following results are also
      a supplement to the literature \cite{SatoH1} if $\psi_t^s\equiv 0$.
 \end{remark}

\subsection{Step size is reduced}
Before the theorem, we need to introduce the following two lemmas. Here,
we first present a lemma that bounds the estimation
error of the estimator.

\begin{lemma}\label{L8}
    Suppose assumption 1.a and assumption 4-7 hold, then
    \begin{equation}\label{L8.1}
        \begin{split}
            &\mathbb{E}[\Vert V_t^s-\text{grad}f(\omega_t^s)\Vert^2\vert \mathcal{F}_t^s]\cr
            \leq&\quad 6(\phi_t^s)^2(M^2+\theta^2N^2)\Vert \xi_{\omega_0^s}^{\omega_t^s}\Vert^2+6(\psi_t^s)^2(M^2+\theta^2N^2)\Vert \xi_{\omega_1}^{\omega_2}\Vert^2\cr
            &+12(1-\phi_t^s-\psi_t^s)^2N^2+(\psi_t^s)^2\Vert V_{t-1}^s-\text{grad}f(\omega_{t-1}^s)\Vert^2\cr
        \end{split}
    \end{equation}
\end{lemma}

Now, we introduce another lemma. The bound produced by the lemma is very important in the proof of the theorem.

\begin{lemma}\label{L9}
    Suppose assumption 1.a, assumption 2 and assumption 4-8 hold. The sequences $\{\omega_t^s\}$
    produced by algorithm 2. Let $\gamma>1, \kappa=\inf\{x\in N^*\vert x^{\gamma}\geq x+1\}$,
    $\alpha_t^s=(t+s+\kappa+2)^{-P}C_\alpha$ and $\psi_t^s=1-(t+s+\kappa+1)^{-Q}C_\psi$,
    satisfying $\max\{\frac{\gamma Q-1}{\gamma-1},\frac{Q}{2}\}\leq P\leq Q$, $0<P,Q<1$. And
    $C_\psi\geq P+C_\alpha^2\cdot 6\beta(M^2+\theta^2N^2)$, where $\beta>2$, for any $2\leq t\leq m-1$ such that
    \begin{eqnarray}\label{L9.1}
        &&\frac{\mathbb{E}[\Vert V_t^s-\text{grad}f(\omega_t^s)\Vert^2]}{\alpha_{t-1}^s}- \frac{\mathbb{E}[\Vert V_{t-1}^s-\text{grad}f(\omega_{t-1}^s)\Vert^2]}{\alpha_{t-2}^s}\cr
        &\leq &\frac{6(M^2+\theta^2N^2)}{\alpha_{t-1}^s}(\phi_t^s)^2\rho^2+12(M^2+\theta^2N^2)\alpha_{t-1}^s \mathbb{E}[\Vert \text{grad}f(\omega_{t-1}^s)\Vert^2]\cr
        &&+6(2-\beta)(M^2+\theta^2N^2)\alpha_{t-1}^s\mathbb{E}[\Vert V_{t-1}^s-\text{grad}f(\omega_{t-1}^s)\Vert^2]\cr
        &&+\frac{12(1-\psi_t^s)^2N^2}{\alpha_{t-1}^s}
    \end{eqnarray}
\end{lemma}

The proofs of lemma $\mathrm{\ref{L8}}$ and $\mathrm{\ref{L9}}$ see Appendix $\mathrm{\ref{secA}}$. The following theorem will introduce the main result about the reduced step size.

\begin{theorem}\label{T5}
    Suppose the conditions in lemma $\mathrm{\ref{L9}}$ are hold, if $C_\alpha \leq \min\{\frac{1}{L},\sqrt{\frac{(1-P)}{6\beta(M^2+\theta^2N^2)}}\}$,
     where $\beta >4$, let $\phi_t^s=(t+s+\kappa+1)^{-R}C_\phi$, $C_\phi \leq C_\psi$ and $R\geq Q$, then
    \begin{equation}\label{T5.1}
        \begin{split}
        &\mathbb{E}[\Vert \text{grad}f(\omega_a)\Vert^2]\cr
        \leq &\frac{\frac{2(\beta-2)}{\beta-4}\Big(\mathbb{E}[f(\tilde{\omega}^{0})-f(\omega^*)]+\sum\limits_{s=1}^{S}\sum\limits_{t=0}^{m-1}\Big(\frac{\rho^2C_\phi^2}{2(\beta-2)C_\alpha}+\frac{12\eta N^2C_\psi^2}{C_\alpha}\Big)(t+s+\kappa+2)^{P-2Q}\big)}{mSC_\alpha(m+S+\kappa+1)^{-P}} \cr
        =&\mathcal{O}(\frac{1}{S^{2(Q-P)}})
        \end{split}
    \end{equation}
    Furthermore, the fastest convergence rate of $\mathbb{E}[\Vert \text{grad}f(\omega_a)\Vert^2]$ at least can get $\mathcal{O}(\frac{1}{S^{\frac{2}{\gamma+1}}})\rightarrow \mathcal{O}(\frac{1}{S})$
\end{theorem}

\begin{remark}
    The inequalities in theorem can get some results and guarantee some conditions hold.
    $C_\alpha \leq \sqrt{\frac{(1-P)}{6\beta(M^2+\theta^2N^2)}}$ can ensure a positive constant $C_{\psi}$ exists, such that $C_{\psi}\in [P+C_\alpha^2\cdot 6\beta(M^2+\theta^2N^2), 1]$, and $C_\alpha < \frac{1}{L}$ implies $ \alpha_t^s<\frac{1}{L}$.
    Moreover, the conditions $C_\phi \leq C_\psi, R\geq Q$ can make $\psi_t^s+\phi_t^s\leq 1$ hold.
    Since assumption 9 is not used, the convergence is global.
\end{remark}

\begin{proof}
    By assumption 8 and $C_\alpha\leq \frac{1}{L}$, we have
    \begin{eqnarray}\label{T5.P.1}
        &&f(\omega_{t+1}^s)\cr
        &\leq& f(\omega_t^s)-\langle \alpha_t^s V_t^s,\text{grad}f(\omega_t^s)\rangle+\frac{(\alpha_t^s)^2L}{2}\Vert V_t^s\Vert^2\cr
        &= &f(\omega_t^s)-\frac{\alpha_t^s}{2}\Vert \text{grad}f(\omega_t^s)\Vert^2-\frac{\alpha_t^s}{2}\Vert V_t^s\Vert^2+\frac{\alpha_t^s}{2}\Vert V_t^s-\text{grad}f(\omega_t^s)\Vert^2+\frac{(\alpha_t^s)^2L}{2}\Vert V_t^s\Vert^2\cr
        &= &f(\omega_t^s)-\frac{\alpha_t^s}{2}\Vert \text{grad}f(\omega_t^s)\Vert^2+\frac{\alpha_t^s}{2}\Vert V_t^s-\text{grad}f(\omega_t^s)\Vert^2+(\frac{(\alpha_t^s)^2L}{2}-\frac{\alpha_t^s}{2})\Vert V_t^s\Vert^2\cr
        &\leq& f(\omega_t^s)-\frac{\alpha_t^s}{2}\Vert \text{grad}f(\omega_t^s)\Vert^2+\frac{\alpha_t^s}{2}\Vert V_t^s-\text{grad}f(\omega_t^s)\Vert^2
        \end{eqnarray}
    According to lemma $\mathrm{\ref{L9}}$, for all $1\leq t\leq m-2$, we obtain
    \begin{eqnarray}\label{T5.P.2}
        &&\frac{1}{12(\beta-2)(M^2+\theta^2N^2)}\Big(\frac{\mathbb{E}[\Vert V_{t+1}^s-\text{grad}f(\omega_{t+1}^s)\Vert^2]}{\alpha_{t}^s}- \frac{\mathbb{E}[\Vert V_{t}^s-\text{grad}f(\omega_{t}^s)\Vert^2]}{\alpha_{t-1}^s}\Big)\cr
        &\leq &\frac{1}{12(\beta-2)(M^2+\theta^2N^2)}\Big(\frac{6(M^2+\theta^2N^2)}{\alpha_{t}^s}(\phi_{t+1}^s)^2\rho^2\cr
        &&+12(M^2+\theta^2N^2)\alpha_{t}^s \mathbb{E}[\Vert \text{grad}f(\omega_{t}^s)\Vert^2]+\frac{12(1-\psi_{t+1}^s)^2N^2}{\alpha_{t}^s}\cr
        &&+6(2-\beta)(M^2+\theta^2N^2)\alpha_{t}^s\mathbb{E}[\Vert V_{t}^s-\text{grad}f(\omega_{t}^s)\Vert^2]\Big)\cr
        &=& \frac{\rho^2}{2(\beta-2)}\cdot \frac{(\phi_{t+1}^s)^2}{\alpha_{t}^s}+\frac{\alpha_t^s}{(\beta-2)}\mathbb{E}[\Vert \text{grad}f(\omega_t^s)\Vert^2]-\frac{\alpha_t^s}{2}\mathbb{E}[\Vert V_t^s-\text{grad}f(\omega_t^s)\Vert^2]\cr
        &&+\frac{N^2}{(\beta-2)(M^2+\theta^2N^2)}\cdot \frac{(1-\psi_{t+1}^s)^2}{\alpha_{t}^s}
    \end{eqnarray}
    Taking the mathematical expectation about $\mathrm{(\ref{T5.P.1})}$ and combining $\mathrm{(\ref{T5.P.2})}$,
    for any $1\leq t\leq m-2$, we get
    \begin{eqnarray}\label{T5.P.3}
        &&\frac{1}{12(\beta-2)(M^2+\theta^2N^2)}\Big(\frac{\mathbb{E}[\Vert V_{t+1}^s-\text{grad}f(\omega_{t+1}^s)\Vert^2]}{\alpha_{t}^s}- \frac{\mathbb{E}[\Vert V_{t}^s-\text{grad}f(\omega_{t}^s)\Vert^2]}{\alpha_{t-1}^s}\Big)\cr
        &&+\mathbb{E}[f(\omega_{t+1}^s)-f(\omega_t^s)]\cr
        &\leq&\frac{\rho^2}{2(\beta-2)}\cdot \frac{(\phi_{t+1}^s)^2}{\alpha_{t}^s}+\frac{\alpha_t^s}{(\beta-2)}\mathbb{E}[\Vert \text{grad}f(\omega_t^s)\Vert^2]-\frac{\alpha_t^s}{2}\mathbb{E}[\Vert V_t^s-\text{grad}f(\omega_t^s)\Vert^2]\cr
        &&+\frac{N^2}{(\beta-2)(M^2+\theta^2N^2)}\cdot \frac{(1-\psi_{t+1}^s)^2}{\alpha_{t}^s} -\frac{\alpha_t^s}{2}\mathbb{E}[\Vert \text{grad}f(\omega_t^s)\Vert^2]\cr
        &&+\frac{\alpha_t^s}{2}\mathbb{E}[\Vert V_t^s-\text{grad}f(\omega_t^s)\Vert^2]\cr
        &\leq&\frac{\rho^2}{2(\beta-2)}\cdot \frac{(\phi_{t+1}^s)^2}{\alpha_{t}^s}+\frac{(4-\beta)\alpha_t^s}{2(\beta-2)}\mathbb{E}[\Vert \text{grad}f(\omega_t^s)\Vert^2]\cr
        &&+\frac{N^2}{(\beta-2)(M^2+\theta^2N^2)}\cdot \frac{(1-\psi_{t+1}^s)^2}{\alpha_{t}^s}
    \end{eqnarray}
    We can definite $V_m^s$ to make above formula is hold for $1\leq t\leq m-1$
    (The $V_m^s$ is exists, for example $V_m^s$ give by algorithm 2).
    Let $\eta=\frac{1}{12(\beta-2)(M^2+\theta^2N^2)}$ and for any $1\leq t\leq m-1$, we have
    \begin{eqnarray}\label{T5.P.4}
        &&\mathbb{E}[f(\omega_{t+1}^s)-f(\omega_t^s)]\cr
        &\leq &\eta \Big(\frac{\mathbb{E}[\Vert V_{t}^s-\text{grad}f(\omega_{t}^s)\Vert^2]}{\alpha_{t-1}^s}-\frac{\mathbb{E}[\Vert V_{t+1}^s-\text{grad}f(\omega_{t+1}^s)\Vert^2]}{\alpha_{t}^s}\Big)\cr
        &&+\frac{\rho^2}{2(\beta-2)}\cdot \frac{(\phi_{t+1}^s)^2}{\alpha_{t}^s}+\frac{(4-\beta)\alpha_t^s}{2(\beta-2)}\mathbb{E}[\Vert \text{grad}f(\omega_t^s)\Vert^2]+12\eta N^2\cdot \frac{(1-\psi_{t+1}^s)^2}{\alpha_{t}^s}
    \end{eqnarray}
    Summing this result over $1\leq t\leq m-1$ gives
    \begin{eqnarray}\label{T5.P.5}
        &&\mathbb{E}[f(\omega_{m}^s)-f(\omega_1^s)]\cr
        &\leq &\eta \Big(\frac{\mathbb{E}[\Vert V_{1}^s-\text{grad}f(\omega_{1}^s)\Vert^2]}{\alpha_{0}^s}-\frac{\mathbb{E}[\Vert V_{m}^s-\text{grad}f(\omega_{m}^s)\Vert^2]}{\alpha_{m-1}^s}\Big)\cr
        &&+\sum\limits_{t=1}^{m-1}\Big(\frac{\rho^2}{2(\beta-2)}\cdot \frac{(\phi_{t+1}^s)^2}{\alpha_{t}^s}+\frac{(4-\beta)\alpha_t^s}{2(\beta-2)}\mathbb{E}[\Vert \text{grad}f(\omega_t^s)\Vert^2]+12\eta N^2\cdot \frac{(1-\psi_{t+1}^s)^2}{\alpha_{t}^s}\Big)
    \end{eqnarray}
    Since $\mathrm{(\ref{L9.P.1})}$, $\mathrm{(\ref{T5.P.1})}$ and $(1-\phi_1^s-\psi_1^s)^2\leq (1-\phi_1^s)^2$ yields
    \begin{eqnarray}\label{T5.P.6}
        &&\mathbb{E}[f(\omega_{1}^s)-f(\omega_0^s)]+\eta \frac{\mathbb{E}[\Vert V_{1}^s-\text{grad}f(\omega_{1}^s)\Vert^2]}{\alpha_{0}^s}\cr
        &\leq& -\frac{\alpha_0^s}{2}\mathbb{E}[\Vert \text{grad}f(\omega_0^s)\Vert^2]+\frac{\rho^2}{2(\beta-2)}\cdot \frac{(\phi_{1}^s)^2}{\alpha_{0}^s}+\frac{\alpha_0^s}{(\beta-2)}\mathbb{E}[\Vert \text{grad}f(\omega_0^s)\Vert^2]\cr
        &&+12\eta N^2\cdot \frac{(1-\psi_{1}^s)^2}{\alpha_{0}^s}\cr
        &=&\frac{\rho^2}{2(\beta-2)}\cdot \frac{(\phi_{1}^s)^2}{\alpha_{0}^s}+\frac{(4-\beta)\alpha_0^s}{2(\beta-2)}\mathbb{E}[\Vert \text{grad}f(\omega_0^s)\Vert^2]+12\eta N^2\cdot \frac{(1-\psi_{1}^s)^2}{\alpha_{0}^s}
    \end{eqnarray}
    Combining the above two inequalities, we get
    \begin{eqnarray}\label{T5.P.7}
        &&\mathbb{E}[f(\omega_{m}^s)-f(\omega_0^s)]\cr
        &=&\mathbb{E}[f(\omega_{1}^s)-f(\omega_0^s)]+\mathbb{E}[f(\omega_{m}^s)-f(\omega_1^s)]\cr
        &\leq &\mathbb{E}[f(\omega_{1}^s)-f(\omega_0^s)]+\eta \Big(\frac{\mathbb{E}[\Vert V_{1}^s-\text{grad}f(\omega_{1}^s)\Vert^2]}{\alpha_{0}^s}-\frac{\mathbb{E}[\Vert V_{m}^s-\text{grad}f(\omega_{m}^s)\Vert^2]}{\alpha_{m-1}^s}\Big)\cr
        &&+\sum\limits_{t=1}^{m-1}\Big(\frac{\rho^2}{2(\beta-2)}\cdot \frac{(\phi_{t+1}^s)^2}{\alpha_{t}^s}+\frac{(4-\beta)\alpha_t^s}{2(\beta-2)}\mathbb{E}[\Vert \text{grad}f(\omega_t^s)\Vert^2]+12\eta N^2\cdot \frac{(1-\psi_{t+1}^s)^2}{\alpha_{t}^s}\Big) \cr
        &\leq &\mathbb{E}[f(\omega_{1}^s)-f(\omega_0^s)]+\eta \frac{\mathbb{E}[\Vert V_{1}^s-\text{grad}f(\omega_{1}^s)\Vert^2]}{\alpha_{0}^s}\cr
        &&+\sum\limits_{t=1}^{m-1}\Big(\frac{\rho^2}{2(\beta-2)}\cdot \frac{(\phi_{t+1}^s)^2}{\alpha_{t}^s}+\frac{(4-\beta)\alpha_t^s}{2(\beta-2)}\mathbb{E}[\Vert \text{grad}f(\omega_t^s)\Vert^2]+12\eta N^2\cdot \frac{(1-\psi_{t+1}^s)^2}{\alpha_{t}^s}\Big) \cr
        &\leq&\sum\limits_{t=0}^{m-1}\Big(\frac{\rho^2}{2(\beta-2)}\cdot \frac{(\phi_{t+1}^s)^2}{\alpha_{t}^s}+\frac{(4-\beta)\alpha_t^s}{2(\beta-2)}\mathbb{E}[\Vert \text{grad}f(\omega_t^s)\Vert^2]+12\eta N^2\cdot \frac{(1-\psi_{t+1}^s)^2}{\alpha_{t}^s}\Big) \cr
        &\leq&\sum\limits_{t=0}^{m-1}\Big(\frac{\rho^2}{2(\beta-2)}\cdot \frac{(\phi_{t+1}^s)^2}{\alpha_{t}^s}+\frac{(4-\beta)\alpha_{m-1}^s}{2(\beta-2)}\mathbb{E}[\Vert \text{grad}f(\omega_t^s)\Vert^2]+12\eta N^2\cdot \frac{(1-\psi_{t+1}^s)^2}{\alpha_{t}^s}\Big)
    \end{eqnarray}
    The second inequality uses the fact $\mathbb{E}[\Vert V_m^s- \text{grad}f(\omega_m^s)\Vert^2]\geq 0$.
    The last inequality is due to the reduced step size , i.e., $\alpha_t^s>\alpha_{m-1}^s$.
    Since $\beta>4$ such that $\frac{(4-\beta)}{2(\beta-2)}<0$, we have
    \begin{eqnarray}\label{T5.P.8}
            &&\alpha_{m-1}^s\sum\limits_{t=0}^{m-1}\mathbb{E}[\Vert \text{grad}f(\omega_t^s)\Vert^2]\cr
            &\leq &\frac{2(\beta-2)}{\beta-4}\Big(\mathbb{E}[f(\omega_{0}^s)-f(\omega_m^s)]+\sum\limits_{t=0}^{m-1}\big(\frac{\rho^2}{2(\beta-2)}\cdot \frac{(\phi_{t+1}^s)^2}{\alpha_{t}^s}+12\eta N^2\cdot \frac{(1-\psi_{t+1}^s)^2}{\alpha_{t}^s}\big)\Big)
    \end{eqnarray}
    Using $\alpha_t^s>\alpha_{m-1}^S$ again, then
    \begin{eqnarray}\label{T5.P.9}
        &&\mathbb{E}[\Vert \text{grad}f(\omega_a)\Vert^2]=\frac{1}{mS}\sum\limits_{s=1}^{S}\sum\limits_{t=0}^{m-1}\mathbb{E}[\Vert \text{grad}f(\omega_t^s)\Vert^2]\cr
        &=&\frac{1}{mS\alpha_{m-1}^S}\alpha_{m-1}^S\sum\limits_{s=1}^{S}\sum\limits_{t=0}^{m-1}\mathbb{E}[\Vert \text{grad}f(\omega_t^s)\Vert^2]\cr
        &\leq&\frac{1}{mS\alpha_{m-1}^S}\sum\limits_{s=1}^{S}\alpha_{m-1}^s\sum\limits_{t=0}^{m-1}\mathbb{E}[\Vert \text{grad}f(\omega_t^s)\Vert^2]\cr
        &\leq&\frac{\frac{2(\beta-2)}{\beta-4}\Big(\mathbb{E}[f(\tilde{\omega}^{0})-f(\omega_m^S)]+\sum\limits_{s=1}^{S}\sum\limits_{t=0}^{m-1}\big(\frac{\rho^2}{2(\beta-2)}\cdot \frac{(\phi_{t+1}^s)^2}{\alpha_{t}^s}+12\eta N^2\cdot \frac{(1-\psi_{t+1}^s)^2}{\alpha_{t}^s}\big)\Big)}{mS\alpha_{m-1}^S}
    \end{eqnarray}
    If $x<0$
    \begin{equation}\label{T5.P.10}
        \sum\limits_{s=1}^{S}\sum\limits_{t=0}^{m-1}(t+s+\kappa+2)^x\leq m\sum\limits_{i=\kappa +2}^{m+S+\kappa+1}i^x\leq m \int_{\kappa+1}^{m+S+\kappa+1}i^x
    \end{equation}
    Thus
    \begin{eqnarray}\label{T5.P.11}
        \sum\limits_{s=1}^{S}\sum\limits_{t=0}^{m-1}(t+s+\kappa+2)^x\leq
        \begin{cases}
            mln(m+S+\kappa+1),\qquad &x=-1\cr
            \frac{m}{x+1}(m+S+\kappa+1)^{x+1},\qquad &-1\leq x\leq 0
        \end{cases}
    \end{eqnarray}
    Since $R\geq Q$, we obtain that $\frac{(\phi_{t+1}^s)^2}{\alpha_{t}^s} \leq \frac{(1-\psi_{t+1}^s)^2}{\alpha_{t}^s}=(t+s+\kappa+2)^{P-2Q}$.
    Giving by $\mathrm{(\ref{T5.P.9})}$-$\mathrm{(\ref{T5.P.11})}$, we have
    \begin{equation}\label{T5.P.12}
        \begin{split}
        &\mathbb{E}[\Vert \text{grad}f(\omega_a)\Vert^2]\cr
        \leq &\frac{\frac{2(\beta-2)}{\beta-4}\Big(\mathbb{E}[f(\tilde{\omega}^{0})-f(\omega^*)]+\sum\limits_{s=1}^{S}\sum\limits_{t=0}^{m-1}\Big(\frac{\rho^2C_\phi^2}{2(\beta-2)C_\alpha}+\frac{12\eta N^2C_\psi^2}{C_\alpha}\Big)(t+s+\kappa+2)^{P-2Q}\big)}{mSC_\alpha(m+S+\kappa+1)^{-P}} \cr
        =&\mathcal{O}(\frac{1}{S^{2(Q-P)}})
        \end{split}
    \end{equation}
    Furthermore, by conditions $\max\{\frac{\gamma Q-1}{\gamma-1},\frac{Q}{2}\}\leq P\leq Q$, $0<P,Q<1$, and
    using linear programming knowledge, we know $2(Q-P)$ get maximum at $Q=\frac{2}{\gamma+1}$ and $P=\frac{1}{\gamma+1}$.
     The maximum value is $[2(Q-P)]_{\max}=\frac{2}{\gamma+1}$,
    By the condition $\gamma>1$, the fastest convergence rate of $\mathbb{E}[\Vert \text{grad}f(\omega_a)\Vert^2]$ at least can get $\mathcal{O}(\frac{1}{S^{\frac{2}{\gamma+1}}})\rightarrow \mathcal{O}(\frac{1}{S})$
\end{proof}

Andi Han et al.\cite{AndiH1} consider the problem of expectation (online)
 minimization over Riemannian manifold $\mathcal{M}$. They assumption
 the stochastic gradient is an unbiased estimation, i.e., $\mathbb{E}_{\omega}\text{grad}f(x,\omega)=\text{grad}F(x)$.
  And they get a convergence rate of $\mathcal{O}(\frac{1}{T^{\frac{2}{3}}})$ for the case of the reduced step size.
 In our paper, we do not need this assumption (similar to lemma $\mathrm{\ref{L1}}$)
 and we can get faster convergence. If $\phi_t^s\equiv 1$ and consider the problem is online, our results can be degenerated into \cite{AndiH1} and more faster.

\subsection{Step size is fixed}
In this subsection, we analyze the case of fixed step size.
The reason why we use lemma $\mathrm{\ref{L8}}$ to proof the theorem, we will elaborate in the remark $\mathrm{(\ref{rmk.T6.1})}$.

\begin{theorem}\label{T6}
    Suppose assumption 1.a, assumption 2 and assumption 4-9 hold. The sequences $\{\omega_t^s\}$
    produced by algorithm 2 with $\alpha_t^s\equiv C_\alpha$ and
    $\psi_t^s=\psi^s$, $\phi_t^s= \phi^s$,
    satisfying
    \begin{eqnarray*}\label{T6.1}
        C_\alpha\leq\frac{2}{L+\sqrt{L^2+4\nu}},\quad \nu=6m^2(M^2+\theta^2N^2)(C_1^2C_2^2m^2+1)
    \end{eqnarray*}
    And $\sum\limits_{s=1}^{\infty}(1-\phi^s-\psi^s)^2<\infty$, then
    \begin{equation}\label{T6.2}
        \begin{split}
            \mathbb{E}[\Vert \text{grad}f(\omega_{a})\Vert^2]=&\frac{1}{mS}\sum\limits_{s=1}^S\sum\limits_{t=0}^{m-1}\mathbb{E}[\Vert \text{grad}f(\omega_{t}^s)\Vert^2]\cr
            \leq&\frac{2}{mC_\alpha S}(f(\tilde{\omega}^0)-f(\omega^*))+\frac{12mN^2}{S}\sum\limits_{s=1}^{\infty}(1-\phi^s-\psi^s)^2\cr
            =&\mathcal{O}(\frac{1}{S})
        \end{split}
    \end{equation}
\end{theorem}
\begin{remark}\label{rmk.T6.1}
    There exists parameters $\phi^s,\psi^s$ satisfy the conditios. For example $\phi^s+\psi^s=1-\frac{1}{s+1}$. 
    We can analyze the convergence similar to theorem $\mathrm{\ref{T3}}$. The result is similar, but the condition is
    $\sum\limits_{s=1}^{\infty}(1-\phi^s)^2+(\psi^s)^2<\infty$. However, if $\phi^s=\psi^s=\frac{1}{2}(1-\frac{1}{s+1})$,
    we can easily find that the parameters satisfy$\sum\limits_{s=1}^{\infty}(1-\phi^s-\psi^s)^2<\infty$
    and not satisfy $\sum\limits_{s=1}^{\infty}(1-\phi^s)^2+(\psi^s)^2<\infty$. Therefore, we will
    use lemma $\mathrm{\ref{L8}}$. The advantage is that we use weaker conditions to obtain similar conclusions.
\end{remark}

\begin{proof}
    According to lemma $\mathrm{\ref{L8}}$, taking the parameters, step size and $\mathrm{(\ref{T3.P.3})}$
    into $\mathrm{(\ref{L8.1})}$, we have
    \begin{eqnarray}\label{T6.P.1}
        &&\mathbb{E}[\Vert V_t^s-\text{grad}f(\omega_t^s)\Vert^2]\cr
        &\leq&12(1-\phi^s-\psi^s)^2N^2+(\psi^s)^2\mathbb{E}[\Vert V_{t-1}^s-\text{grad}f(\omega_{t-1}^s)\Vert^2]\cr
        &&+ 6C_\alpha^2(M^2+\theta^2N^2)\Big(C_1^2C_2^2(\phi^s)^2t\sum\limits_{i=0}^{t-1} \mathbb{E}[\Vert V_i^s\Vert^2]+(\psi^s)^2\mathbb{E}[\Vert V_{t-1}^s\Vert^2]\Big)\cr
        &\leq&12(1-\phi^s-\psi^s)^2N^2+\mathbb{E}[\Vert V_{t-1}^s-\text{grad}f(\omega_{t-1}^s)\Vert^2]\cr
        &&+ 6C_\alpha^2(M^2+\theta^2N^2)\Big(C_1^2C_2^2(\phi^s)^2m\sum\limits_{t=0}^{m-1} \mathbb{E}[\Vert V_t^s\Vert^2]+(\psi^s)^2 \mathbb{E}[\Vert V_t^s\Vert^2]\Big)\cr
        &\leq&12t(1-\phi^s-\psi^s)^2N^2+\mathbb{E}[\Vert V_0^s-\text{grad}f(\omega_0^s)\Vert^2]\cr
        &&+ 6C_\alpha^2(M^2+\theta^2N^2)\Big(C_1^2C_2^2(\phi^s)^2mt\sum\limits_{t=0}^{m-1} \mathbb{E}[\Vert V_t^s\Vert^2]+(\psi^s)^2 \sum\limits_{i=0}^{t-1} \mathbb{E}[\Vert V_i^s\Vert^2]\Big)\cr
        &\leq&12m(1-\phi^s-\psi^s)^2N^2+\mathbb{E}[\Vert V_0^s-\text{grad}f(\omega_0^s)\Vert^2]\cr
        &&+6C_\alpha^2(M^2+\theta^2N^2)\Big(C_1^2C_2^2(\phi^s)^2m^2\sum\limits_{t=0}^{m-1} \mathbb{E}[\Vert V_t^s\Vert^2]+(\psi^s)^2\sum\limits_{t=0}^{m-1} \mathbb{E}[\Vert V_t^s\Vert^2]\Big)\cr
        &=&6C_\alpha^2(M^2+\theta^2N^2)\Big(C_1^2C_2^2(\phi^s)^2m^2+(\psi^s)^2)\sum\limits_{t=0}^{m-1} \mathbb{E}[\Vert V_t^s\Vert^2]\Big)\cr
        &&+12m(1-\phi^s-\psi^s)^2N^2
    \end{eqnarray}
    The above inequality applies $V_0^s=\text{grad}f(\omega_0^s)$ and $t\leq m$. Summing this result over $t=0,...,m-1$ and  $ s=1,...,S$ gives
    \begin{eqnarray}\label{T6.P.2}
        &&\sum\limits_{s=1}^S\sum\limits_{t=0}^{m-1}\mathbb{E}[\Vert V_t^s-\text{grad}f(\omega_t^s)\Vert^2]\cr
        &\leq&\sum\limits_{s=1}^S\Big[6mC_\alpha^2(M^2+\theta^2N^2)\Big(C_1^2C_2^2(\phi^s)^2m^2+(\psi^s)^2\Big)\sum\limits_{t=0}^{m-1} \mathbb{E}[\Vert V_t^s\Vert^2\cr
        &&+12N^2m^2(1-\phi^s-\psi^s)^2]\Big]\cr
        &\leq&\sum\limits_{s=1}^S\Big[6mC_\alpha^2(M^2+\theta^2N^2)\Big(C_1^2C_2^2m^2+1\Big)\sum\limits_{t=0}^{m-1} \mathbb{E}[\Vert V_t^s\Vert^2+12N^2m^2(1-\phi^s-\psi^s)^2]\Big]\cr
        &\leq&6mC_\alpha^2(M^2+\theta^2N^2)\Big(C_1^2C_2^2m^2+1\Big)\sum\limits_{s=1}^S\sum\limits_{t=0}^{m-1} \mathbb{E}[\Vert V_t^s\Vert^2]\cr
        &&+12N^2m^2(1-\phi^s-\psi^s)^2\sum\limits_{s=1}^{\infty}(1-\phi^s-\psi^s)^2+
    \end{eqnarray}
    By the condition in assumption 8, summing with $\mathrm{(\ref{T3.P.2})}$, we can easily verify that
    \begin{eqnarray}\label{T6.P.3}
        &&\sum\limits_{s=1}^S\sum\limits_{t=0}^{m-1}\mathbb{E}[\Vert \text{grad}f(\omega_{t}^s)\Vert^2] \cr
        &\leq& \frac{2}{C_\alpha}\mathbb{E}[f(\tilde{\omega}^0)-f(\omega_m^s)]+\sum\limits_{s=1}^S\sum\limits_{t=0}^{m-1}\mathbb{E}[\Vert V_t^s-\text{grad}f(\omega_{t}^s)\Vert^2]+(LC_\alpha-1)\sum\limits_{s=1}^S\sum\limits_{t=0}^{m-1}\mathbb{E}[\Vert V_t^s\Vert^2]\cr
        &\leq& \frac{2}{C_\alpha}\mathbb{E}[f(\tilde{\omega}^0)-f(\omega_m^s)]+12N^2m^2\sum\limits_{s=1}^{\infty}(1-\phi^s-\psi^s)^2\cr
        &&+\Big[6mC_\alpha^2(M^2+\theta^2N^2)\Big(C_1^2C_2^2m^2+1\Big)+LC_\alpha-1\Big]\sum\limits_{s=1}^S\sum\limits_{t=0}^{m-1} \mathbb{E}[\Vert V_t^s\Vert^2]\cr
        &\leq& \frac{2}{C_\alpha}\mathbb{E}[f(\tilde{\omega}^0)-f(\omega_m^s)]+12N^2m^2\sum\limits_{s=1}^{\infty}(1-\phi^s-\psi^s)^2\cr
        &\leq& \frac{2}{C_\alpha}(f(\tilde{\omega}^0)-f(\omega^*))+12N^2m^2\sum\limits_{s=1}^{\infty}(1-\phi^s-\psi^s)^2
    \end{eqnarray}
    The third inequality is based the fact that $6mC_\alpha^2(M^2+\theta^2N^2)\Big(C_1^2C_2^2m^2+1\Big)+LC_\alpha-1\leq 0$, if $C_\alpha\leq \frac{-L+\sqrt{L^2+4\nu}}{2\nu}=\frac{2}{L+\sqrt{L^2+4\nu}}$.
    Thus
    \begin{eqnarray}\label{T6.P.4}
        \mathbb{E}[\Vert \text{grad}f(\omega_{a})\Vert^2]&=&\frac{1}{mS}\sum\limits_{s=1}^S\sum\limits_{t=0}^{m-1}\mathbb{E}[\Vert \text{grad}f(\omega_{t}^s)\Vert^2]\cr
        &\leq&\frac{2}{mC_\alpha S}(f(\tilde{\omega}^0)-f(\omega^*))+\frac{12mN^2}{S}\sum\limits_{s=1}^{\infty}(1-\phi^s-\psi^s)^2\cr
        &=&\mathcal{O}(\frac{1}{S})
    \end{eqnarray}
\end{proof}

\subsection{Special case}
We now turn to a special case of problem (P) with $\tau-$gradient dominated function.
As an important class of non-convex function, we can establish linear convergence
for this non-convex functions. Here, we only consider this special case, 
and other special cases are similar to above section, so we will not considere in this subsection.

\begin{theorem}
    Suppose the conditions in theorem $\mathrm{\ref{T6}}$ are hold.
    If $\psi^s+\phi^s=1$, $S=\lceil\frac{2\tau \gamma}{mC_\alpha}\rceil$, $\gamma>1$ and $\tilde{\omega}^{k+1}=\text{Alg2}(\tilde{\omega}^{k},\tilde{\omega}^{0},m,S,\psi^s,\phi^s)$,
    $0\leq k\leq K-1$, the function $f$ is a $\tau$-gradient dominated functions. Then
    \begin{equation}
        \begin{split}
            \mathbb{E}[\Vert \text{grad}f(\tilde{\omega}^{K})\Vert^2] &\leq \gamma^{-K}\mathbb{E}[\Vert \text{grad}f(\tilde{\omega}^{0})\Vert^2]\cr\
            \mathbb{E}[(f(\tilde{\omega}^K)-f(\omega^*))]&\leq \gamma^{-K}\mathbb{E}[(f(\tilde{\omega}^0)-f(\omega^*))]
        \end{split}
    \end{equation}
    Furthermore, we obtain $\lim\limits_{K\rightarrow \infty}\mathbb{E}[\Vert\text{grad}f(\tilde{\omega}^{K})\Vert]=0$ and $\lim\limits_{K\rightarrow \infty} \mathbb{E}[(f(\tilde{\omega}^K)]=f(\omega^*))$.
\end{theorem}

\begin{proof}
     By the condition $S=\lceil\frac{2\tau \gamma}{mC_\alpha}\rceil$, we have $\frac{2\tau}{mC_\alpha S}\leq \frac{1}{\gamma}$.
    From theorem $\mathrm{\ref{T6}}$, for any  $1\leq k\leq K-1$, we have
    \begin{equation}
        \begin{split}
            &\mathbb{E}[\Vert \text{grad}f(\tilde{\omega}^{k+1})\Vert^2]
            \leq \frac{2}{mC_\alpha S}\mathbb{E}[(f(\tilde{\omega}^k)-f(\omega^*))]\cr
             \leq &\frac{2\tau}{mC_\alpha S}\mathbb{E}[\Vert \text{grad}f(\tilde{\omega}^{k})\Vert^2]\leq \frac{1}{\gamma}\mathbb{E}[\Vert \text{grad}f(\tilde{\omega}^{k})\Vert^2]\cr
             \leq &\frac{1}{\gamma}\cdot \frac{2}{mC_\alpha S}\mathbb{E}[(f(\tilde{\omega}^{k-1})-f(\omega^*))]\cr
        \end{split}
    \end{equation}
    The second inequality is due to $f$ is a $\tau$-gradient dominated functions.
    Iterate on both sides of the inequality, we get
    \begin{equation}
        \begin{split}
            &\mathbb{E}[\Vert \text{grad}f(\tilde{\omega}^{k+1})\Vert^2]\leq \frac{1}{\gamma}\mathbb{E}[\Vert \text{grad}f(\tilde{\omega}^{k})\Vert^2] \leq \cdots \leq \frac{1}{\gamma^{k+1}}\mathbb{E}[\Vert \text{grad}f(\tilde{\omega}^{0})\Vert^2]\cr\
            &\mathbb{E}[(f(\tilde{\omega}^k)-f(\omega^*))]\leq \frac{1}{\gamma}\mathbb{E}[(f(\tilde{\omega}^{k-1})-f(\omega^*))]\leq \cdots \leq \frac{1}{\gamma^{k}}\mathbb{E}[(f(\tilde{\omega}^0)-f(\omega^*))]
        \end{split}
    \end{equation}
\end{proof}

\section{Conclusions}
This paper proposes R-SHG algorithm with adaptive parameters and time-varying parameters by the linear combination of R-SRG, R-SVRG and R-SGD. We have studied the finite-sum optimization problems on a smooth Riemannian manifold $\mathcal{M}$. Two R-SHG algorithms with two different step sizes have been considered. Compared to the existing literature, our model is more widely applicable in the sense that 1) we do not need the descent direction to be an unbiased estimate; 2) our analysis focuses on retraction mapping and vector transport, do not need exponential mapping or vector transport. At the algorithm of R-SHG with adaptive parameters and time-varying parameters, we get global convergence when the step size is reduced and quantitatively research the convergence when the step size is fixed. For some special cases, we give better results. In this paper, there is no special requirement for function $f$. Next, we will research whether the function satisfying certain conditions can have better properties and consider adaptive batch size gradient of a reference point.

\begin{appendices}

\noindent
\section{Proofs of lemmas in section 3 and section 4}\label{secA}
\textbf{Proof of Lemma $\mathrm{\ref{L0}}$}
\begin{proof}
The inequalities are discussed in two cases\\
If $\langle \mathcal{T}_{\omega_{t-1}^s}^{\omega_t^s} \big(V_{t-1}^s-\text{grad}f(\omega_{t-1}^s)\big),\text{grad}f(\omega_t^s)\rangle>0$, then
\begin{eqnarray}
    &0\geq& \tilde{\psi}_t^s\langle \mathcal{T}_{\omega_{t-1}^s}^{\omega_t^s} \big(V_{t-1}^s-\text{grad}f(\omega_{t-1}^s)\big),\text{grad}f(\omega_t^s)\rangle\cr
    &=& \min\{\psi_t^s,\frac{\mu\Vert \text{grad}f(\omega_t^s)\Vert^2}{\vert\langle \mathcal{T}_{\omega_{t-1}^s}^{\omega_t^s} (V_{t-1}^s-\text{grad}f(\omega_{t-1}^s)),\text{grad}f(\omega_t^s)\rangle\vert}\}\cr
    &&\times \langle \mathcal{T}_{\omega_{t-1}^s}^{\omega_t^s} \big(V_{t-1}^s-\text{grad}f(\omega_{t-1}^s)\big),\text{grad}f(\omega_t^s)\rangle\cr
    &\leq &\mu\Vert \text{grad}f(\omega_t^s)\Vert^2
\end{eqnarray}
If $\langle \mathcal{T}_{\omega_{t-1}^s}^{\omega_t^s} \big(V_{t-1}^s-\text{grad}f(\omega_{t-1}^s)\big),\text{grad}f(\omega_t^s)\rangle<0$, then
\begin{eqnarray}
    &0\leq& \tilde{\psi}_t^s\langle \mathcal{T}_{\omega_{t-1}^s}^{\omega_t^s} \big(V_{t-1}^s-\text{grad}f(\omega_{t-1}^s)\big),\text{grad}f(\omega_t^s)\rangle\cr
    &=& \min\{\psi_t^s,\frac{\mu\Vert \text{grad}f(\omega_t^s)\Vert^2}{\vert\langle \mathcal{T}_{\omega_{t-1}^s}^{\omega_t^s} (V_{t-1}^s-\text{grad}f(\omega_{t-1}^s)),\text{grad}f(\omega_t^s)\rangle\vert}\}\cr
    &&\times \langle \mathcal{T}_{\omega_{t-1}^s}^{\omega_t^s} \big(V_{t-1}^s-\text{grad}f(\omega_{t-1}^s)\big),\text{grad}f(\omega_t^s)\rangle\cr
    &=& -\min\{\psi_t^s,\frac{\mu\Vert \text{grad}f(\omega_t^s)\Vert^2}{\vert\langle \mathcal{T}_{\omega_{t-1}^s}^{\omega_t^s} (V_{t-1}^s-\text{grad}f(\omega_{t-1}^s)),\text{grad}f(\omega_t^s)\rangle\vert}\}\cr
    &&\times -\Big(\langle \mathcal{T}_{\omega_{t-1}^s}^{\omega_t^s} \big(V_{t-1}^s-\text{grad}f(\omega_{t-1}^s)\big),\text{grad}f(\omega_t^s)\rangle\Big)\cr
    &=& \max\{-\psi_t^s,-\frac{\mu\Vert \text{grad}f(\omega_t^s)\Vert^2}{\vert\langle \mathcal{T}_{\omega_{t-1}^s}^{\omega_t^s} (V_{t-1}^s-\text{grad}f(\omega_{t-1}^s)),\text{grad}f(\omega_t^s)\rangle\vert}\}\cr
    & &\times \Big(-\langle \mathcal{T}_{\omega_{t-1}^s}^{\omega_t^s} \big(V_{t-1}^s-\text{grad}f(\omega_{t-1}^s)\big),\text{grad}f(\omega_t^s)\rangle\Big)\cr
    &=& \max\{\psi_t^s\times \langle \mathcal{T}_{\omega_{t-1}^s}^{\omega_t^s} \big(V_{t-1}^s-\text{grad}f(\omega_{t-1}^s),\text{grad}f(\omega_t^s)\rangle\Big),-\mu\Vert \text{grad}f(\omega_t^s)\Vert^2\}\cr
    &\geq &-\mu\Vert \text{grad}f(\omega_t^s)\Vert^2
\end{eqnarray}
\end{proof}

\noindent
\textbf{Proof of Lemma $\mathrm{\ref{L5}}$}
\begin{proof}
    \begin{eqnarray} \label{L5.P.1}
    &&\mathbb{E}[\Vert \text{grad}f_{I_t^s}(\omega_2)-\mathcal{T}_{\omega_1}^{\omega_2} \text{grad}f_{I_t^s}(\omega_1)\Vert^2\vert \mathcal{F}_t^s]\cr
    & =&\mathbb{E}[\Vert\frac{1}{b}\sum_{i\in I_t^s} \text{grad}f_i(\omega_2)-\mathcal{T}_{\omega_1}^{\omega_2} \text{grad}f_i(\omega_1)\Vert^2\vert \mathcal{F}_t^s]\cr
    &\leq &\frac{1}{b}\sum_{i\in I_t^s}\mathbb{E}[\Vert \text{grad}f_i(\omega_2)-\mathcal{T}_{\omega_1}^{\omega_2} \text{grad}f_i(\omega_1)\Vert^2\vert \mathcal{F}_t^s]\cr
    &= &\frac{1}{n}\sum_{i=1}^n\mathbb{E}[\Vert \text{grad}f_i(\omega_2)-\mathcal{T}_{\omega_1}^{\omega_2} \text{grad}f_i(\omega_1)\Vert^2\vert \mathcal{F}_t^s]\cr
    &\leq &\frac{2}{n}\sum_{i=1}^n(\Vert \text{grad}f_i(\omega_2)-\Gamma_{\omega_1}^{\omega_2} \text{grad}f_i(\omega_1)\Vert^2+\Vert\Gamma_{\omega_1}^{\omega_2} \text{grad}f_i(\omega_1)-\mathcal{T}_{\omega_1}^{\omega_2} \text{grad}f_i(\omega_1)\Vert^2)\cr
    &\leq &\frac{2}{n}\sum_{i=1}^n(M^2\Vert \xi\Vert^2+\theta^2\Vert \text{grad}f_i(\omega_1)\Vert^2\Vert \xi\Vert^2)\cr
    &\leq &\frac{2}{n}\sum_{i=1}^n(M^2+\theta^2N^2)\Vert \xi_{\omega_1}^{\omega_2}\Vert^2\cr
    &=& 2(M^2+\theta^2N^2)\Vert \xi_{\omega_1}^{\omega_2}\Vert^2
    \end{eqnarray}
\end{proof}

\noindent
\textbf{Proof of Lemma $\mathrm{\ref{L6}}$}
\begin{proof}
    By the definition of $V_t^s$, we have
    \begin{eqnarray}\label{L6.P.1}
        &&\mathbb{E}[\Vert V_t^s-\text{grad}f(\omega_t^s)\Vert^2\vert \mathcal{F}_t^s]\cr
        &=&\mathbb{E}[\Vert \phi_t^s\Big(\text{grad}f_{I_t^s}(\omega_t^s)-\mathcal{T}_{\omega_0^s}^{\omega_t^s} \big(\text{grad}f_{I_t^s}(\omega_0^s)-\text{grad}f(\omega_0^s)\big) \Big)\cr
        &&+\tilde{\psi}_t^s\Big(\text{grad}f_{I_t^s}(\omega_t^s)-\mathcal{T}_{\omega_{t-1}^s}^{\omega_t^s} \big(\text{grad}f_{I_t^s}(\omega_{t-1}^s)-V_{t-1}^s\big) \Big)\cr
        &&+(1-\phi_t^s-\tilde{\psi}_t^s)\text{grad}f_{I_t^s}(\omega_t^s)-\text{grad}f(\omega_t^s)\Vert^2\vert \mathcal{F}_t^s]\cr
        &=&\mathbb{E}[\Vert\phi_t^s\Big(\text{grad}f_{I_t^s}(\omega_t^s)-\mathcal{T}_{\omega_0^s}^{\omega_t^s} \big(\text{grad}f_{I_t^s}(\omega_0^s)-\text{grad}f(\omega_0^s)\big)-\text{grad}f(\omega_t^s)\Big)\cr
        &&+(1-\phi_t^s)\text{grad}f_{I_t^s}(\omega_t^s)-\tilde{\psi}_t^s\mathcal{T}_{\omega_{t-1}^s}^{\omega_t^s} \big(\text{grad}f_{I_t^s}(\omega_{t-1}^s)-V_{t-1}^s\big)-(1-\phi_t^s)\text{grad}f(\omega_t^s)\Vert^2\vert \mathcal{F}_t^s]\cr
        &=&\mathbb{E}[\Vert\phi_t^s\Big(\text{grad}f_{I_t^s}(\omega_t^s)-\mathcal{T}_{\omega_0^s}^{\omega_t^s} \big(\text{grad}f_{I_t^s}(\omega_0^s)-\text{grad}f(\omega_0^s)\big)-\text{grad}f(\omega_t^s)\Big)\cr
        &&+(1-\phi_t^s)\text{grad}f_{I_t^s}(\omega_t^s)-\tilde{\psi}_t^s\mathcal{T}_{\omega_{t-1}^s}^{\omega_t^s}\text{grad}f_{I_t^s}(\omega_{t-1}^s)\cr
        &&-((1-\phi_t^s)\text{grad}f(\omega_t^s)-\tilde{\psi}_t^s\mathcal{T}_{\omega_{t-1}^s}^{\omega_t^s}\text{grad}f(\omega_{t-1}^s))\cr
        &&+\tilde{\psi}_t^s\mathcal{T}_{\omega_{t-1}^s}^{\omega_t^s}V_{t-1}^s-\tilde{\psi}_t^s\mathcal{T}_{\omega_{t-1}^s}^{\omega_t^s}\text{grad}f(\omega_{t-1}^s)\Vert^2\vert \mathcal{F}_t^s]\cr
        &=&\mathbb{E}[\Vert\phi_t^s\Big(\text{grad}f_{I_t^s}(\omega_t^s)-\mathcal{T}_{\omega_0^s}^{\omega_t^s} \big(\text{grad}f_{I_t^s}(\omega_0^s)-\text{grad}f(\omega_0^s)\big)-\text{grad}f(\omega_t^s)\Big)\cr
        &&+(1-\phi_t^s)\text{grad}f_{I_t^s}(\omega_t^s)-\tilde{\psi}_t^s\mathcal{T}_{\omega_{t-1}^s}^{\omega_t^s}\text{grad}f_{I_t^s}(\omega_{t-1}^s)\cr
        &&-((1-\phi_t^s)\text{grad}f(\omega_t^s)-\tilde{\psi}_t^s\mathcal{T}_{\omega_{t-1}^s}^{\omega_t^s}\text{grad}f(\omega_{t-1}^s))\Vert^2\vert \mathcal{F}_t^s]\cr
        &&+\mathbb{E}[\Vert\tilde{\psi}_t^s\mathcal{T}_{\omega_{t-1}^s}^{\omega_t^s}V_{t-1}^s-\tilde{\psi}_t^s\mathcal{T}_{\omega_{t-1}^s}^{\omega_t^s}\text{grad}f(\omega_{t-1}^s)\Vert^2\vert \mathcal{F}_t^s]\cr
        &&+\mathbb{E}[\big\langle\phi_t^s\Big(\text{grad}f_{I_t^s}(\omega_t^s)-\mathcal{T}_{\omega_0^s}^{\omega_t^s} \big(\text{grad}f_{I_t^s}(\omega_0^s)-\text{grad}f(\omega_0^s)\big)-\text{grad}f(\omega_t^s)\Big)\cr
        &&+(1-\phi_t^s)\text{grad}f_{I_t^s}(\omega_t^s)-\tilde{\psi}_t^s\mathcal{T}_{\omega_{t-1}^s}^{\omega_t^s}\text{grad}f_{I_t^s}(\omega_{t-1}^s)\cr
        &&-((1-\phi_t^s)\text{grad}f(\omega_t^s)-\tilde{\psi}_t^s\mathcal{T}_{\omega_{t-1}^s}^{\omega_t^s}\text{grad}f(\omega_{t-1}^s)),\cr
        &&\tilde{\psi}_t^s\mathcal{T}_{\omega_{t-1}^s}^{\omega_t^s}V_{t-1}^s-\tilde{\psi}_t^s\mathcal{T}_{\omega_{t-1}^s}^{\omega_t^s}\text{grad}f(\omega_{t-1}^s)\big\rangle\vert \mathcal{F}_t^s]\cr
        &=&\mathbb{E}[\Vert\phi_t^s\Big(\text{grad}f_{I_t^s}(\omega_t^s)-\mathcal{T}_{\omega_0^s}^{\omega_t^s} \big(\text{grad}f_{I_t^s}(\omega_0^s)-\text{grad}f(\omega_0^s)\big)-\text{grad}f(\omega_t^s)\Big)\cr
        &&+(1-\phi_t^s)\text{grad}f_{I_t^s}(\omega_t^s)-\tilde{\psi}_t^s\mathcal{T}_{\omega_{t-1}^s}^{\omega_t^s}\text{grad}f_{I_t^s}(\omega_{t-1}^s)\cr
        &&-((1-\phi_t^s)\text{grad}f(\omega_t^s)-\tilde{\psi}_t^s\mathcal{T}_{\omega_{t-1}^s}^{\omega_t^s}\text{grad}f(\omega_{t-1}^s))\Vert^2\vert \mathcal{F}_t^s]\cr
        &&+\mathbb{E}[\Vert\tilde{\psi}_t^s\mathcal{T}_{\omega_{t-1}^s}^{\omega_t^s}V_{t-1}^s-\tilde{\psi}_t^s\mathcal{T}_{\omega_{t-1}^s}^{\omega_t^s}\text{grad}f(\omega_{t-1}^s)\Vert^2\vert \mathcal{F}_t^s]\cr
        &\leq&2(\phi_t^s)^2\mathbb{E}[\Vert\text{grad}f_{I_t^s}(\omega_t^s)-\mathcal{T}_{\omega_0^s}^{\omega_t^s} \big(\text{grad}f_{I_t^s}(\omega_0^s)-\text{grad}f(\omega_0^s)\big)-\text{grad}f(\omega_t^s)\vert \mathcal{F}_t^s]\cr
        &&+2\mathbb{E}[\Vert(1-\phi_t^s)\text{grad}f_{I_t^s}(\omega_t^s)-\tilde{\psi}_t^s\mathcal{T}_{\omega_{t-1}^s}^{\omega_t^s}\text{grad}f_{I_t^s}(\omega_{t-1}^s)\cr
        &&-((1-\phi_t^s)\text{grad}f(\omega_t^s)-\tilde{\psi}_t^s\mathcal{T}_{\omega_{t-1}^s}^{\omega_t^s}\text{grad}f(\omega_{t-1}^s))\vert \mathcal{F}_t^s]\cr
        &&+\mathbb{E}[\Vert\tilde{\psi}_t^s\mathcal{T}_{\omega_{t-1}^s}^{\omega_t^s}V_{t-1}^s-\tilde{\psi}_t^s\mathcal{T}_{\omega_{t-1}^s}^{\omega_t^s}\text{grad}f(\omega_{t-1}^s)\Vert^2\vert \mathcal{F}_t^s]
    \end{eqnarray}
    The fifth inequality is due to $\mathbb{E}[(1-\psi_t^s)\text{grad}f_{I_t^s}(\omega_t^s)-\phi_t^s\mathcal{T}_{\omega_0^s}^{\omega_t^s} \big(\text{grad}f_{I_t^s}(\omega_0^s)-\text{grad}f(\omega_0^s)\big)-(1-\psi_t^s)\text{grad}f(\omega_t^s)
    +\psi_t^s\Big(\text{grad}f_{I_t^s}(\omega_t^s)-\mathcal{T}_{\omega_{t-1}^s}^{\omega_t^s}\text{grad}f_{I_t^s}(\omega_{t-1}^s)-(\text{grad}f(\omega_t^s)-\mathcal{T}_{\omega_{t-1}^s}^{\omega_t^s}\text{grad}f(\omega_{t-1}^s))\Big)\vert \mathcal{F}_t^s]=0$
    and $\mathcal{T}_{\omega_{t-1}^s}^{\omega_t^s}V_{t-1}^s-\mathcal{T}_{\omega_{t-1}^s}^{\omega_t^s}\text{grad}f(\omega_{t-1}^s)$ is measurable in $\mathcal{F}_t^s$.
    The firstly inequality applies $(a+b)^2\leq 2(a^2+b^2)$.
    Now, we consider the each term on the right side of $\mathrm{(\ref{L6.P.1})}$.
    For the first item at the right side of $\mathrm{(\ref{L6.P.1})}$
    \begin{eqnarray}\label{L6.P.2}
        &&\mathbb{E}[\Vert\text{grad}f_{I_t^s}(\omega_t^s)-\mathcal{T}_{\omega_0^s}^{\omega_t^s} \big(\text{grad}f_{I_t^s}(\omega_0^s)-\text{grad}f(\omega_0^s)\big)-\text{grad}f(\omega_t^s)\Vert^2\vert \mathcal{F}_t^s]\cr
        &= &\mathbb{E}[\Vert\text{grad}f_{I_t^s}(\omega_t^s)-\mathcal{T}_{\omega_0^s}^{\omega_t^s}\text{grad}f_{I_t^s}(\omega_0^s)-(\text{grad}f(\omega_t^s)-\mathcal{T}_{\omega_0^s}^{\omega_t^s}\text{grad}f(\omega_0^s))\Vert^2\vert \mathcal{F}_t^s]\cr
        &\leq &\mathbb{E}[\text{grad}f_{I_t^s}(\omega_t^s)-\mathcal{T}_{\omega_0^s}^{\omega_t^s}\text{grad}f_{I_t^s}(\omega_0^s)\Vert^2\vert \mathcal{F}_t^s]\cr
        &\leq &2(M^2+\theta^2N^2)\Vert \xi_{\omega_0^s}^{\omega_t^s}\Vert^2
    \end{eqnarray}
    The firstly inequality follows from $\mathbb{E}[\Vert x-\mathbb{E}[x]\Vert^2\vert \mathcal{F}_t^s]\leq \mathbb{E}\Vert[x\Vert^2\vert \mathcal{F}_t^s]$,
    The second inequality are from $\mathrm{(\ref{L5.1})}$.
    For the second item at the right side of $\mathrm{(\ref{L6.P.1})}$, 
    similar to the proof of $\mathrm{(\ref{L5.P.1})}$, it is easy to find that
    \begin{eqnarray}\label{L6.P.3}
      &&\mathbb{E}[\Vert(1-\phi_t^s)\text{grad}f_{I_t^s}(\omega_t^s)-\tilde{\psi}_t^s\mathcal{T}_{\omega_{t-1}^s}^{\omega_t^s}\text{grad}f_{I_t^s}(\omega_{t-1}^s)\cr
      &&-((1-\phi_t^s)\text{grad}f(\omega_t^s)-\tilde{\psi}_t^s\mathcal{T}_{\omega_{t-1}^s}^{\omega_t^s}\text{grad}f(\omega_{t-1}^s))\Vert^2\vert \mathcal{F}_t^s]\cr
      &\leq&\mathbb{E}[\Vert(1-\phi_t^s)\text{grad}f_{I_t^s}(\omega_t^s)-\tilde{\psi}_t^s\mathcal{T}_{\omega_{t-1}^s}^{\omega_t^s}\text{grad}f_{I_t^s}(\omega_{t-1}^s)\Vert^2\vert \mathcal{F}_t^s]\cr
      &\leq & \frac{1}{n}\sum\limits_{i=1}^{n}\Vert(1-\phi_t^s)\text{grad}f_i(\omega_t^s)-\tilde{\psi}_t^s\mathcal{T}_{\omega_{t-1}^s}^{\omega_t^s}\text{grad}f_i(\omega_{t-1}^s)\Vert^2\cr
      &\leq& \frac{2}{n}\sum\limits_{i=1}^{n}\big(\Vert(1-\phi_t^s)\text{grad}f_i(\omega_t^s)\Vert^2+\Vert\tilde{\psi}_t^s\mathcal{T}_{\omega_{t-1}^s}^{\omega_t^s}\text{grad}f_i(\omega_{t-1}^s)\Vert^2\big)\cr
      &\leq& \frac{2}{n}\sum\limits_{i=1}^{n}\big((1-\phi_t^s)^2N^2+(\tilde{\psi}_t^s)^2N^2\big)\cr
      &=& 2N^2\big((1-\phi_t^s)^2+(\tilde{\psi}_t^s)^2\big)\cr
      &\leq& 2N^2\big((1-\phi_t^s)^2+(\psi_t^s)^2\big)
    \end{eqnarray}
    For the third item at the right side of $\mathrm{(\ref{L6.P.1})}$, using assumption 4, we obtain
    \begin{eqnarray}\label{L6.P.4}
            &&\mathbb{E}[\Vert\tilde{\psi}_t^s\mathcal{T}_{\omega_{t-1}^s}^{\omega_t^s}V_{t-1}^s-\tilde{\psi}_t^s\mathcal{T}_{\omega_{t-1}^s}^{\omega_t^s}\text{grad}f(\omega_{t-1}^s)\Vert^2\vert \mathcal{F}_t^s]\cr
            &=&(\tilde{\psi}_t^s)^2\mathbb{E}[\Vert\mathcal{T}_{\omega_{t-1}^s}^{\omega_t^s}V_{t-1}^s-\mathcal{T}_{\omega_{t-1}^s}^{\omega_t^s}\text{grad}f(\omega_{t-1}^s)\Vert^2\vert \mathcal{F}_t^s]\cr
            &=&(\tilde{\psi}_t^s)^2\mathbb{E}[\Vert V_{t-1}^s-\text{grad}f(\omega_{t-1}^s)\Vert^2\vert \mathcal{F}_t^s]\cr
            &=&(\tilde{\psi}_t^s)^2\Vert V_{t-1}^s-\text{grad}f(\omega_{t-1}^s)\Vert^2\cr
            &\leq &(\psi_t^s)^2\Vert V_{t-1}^s-\text{grad}f(\omega_{t-1}^s)\Vert^2
    \end{eqnarray}
    Combining the inequalities $\mathrm{(\ref{L6.P.1})}$-$\mathrm{(\ref{L6.P.4})}$, we get
    \begin{eqnarray}\label{L6.P.5}
            &&\mathbb{E}[\Vert V_t^s-\text{grad}f(\omega_t^s)\Vert^2\vert \mathcal{F}_t^s]\cr
           &\leq &4(M^2+\theta^2N^2)(\phi_t^s)^2\Vert \xi_{\omega_0^s}^{\omega_t^s}\Vert^2+4N^2\big((1-\phi_t^s)^2+(\psi_t^s)^2\big)\cr
           &&+(\psi_t^s)^2\Vert V_{t-1}^s-\text{grad}f(\omega_{t-1}^s)\Vert^2
    \end{eqnarray}
\end{proof}

\noindent
\textbf{Proof of Lemma $\mathrm{\ref{L7}}$}
\begin{proof}
    By the definition of $V_t^s$, we get
    \begin{eqnarray}\label{L7.P.1}
        &&\mathbb{E}[\Vert V_t^s-\text{grad}f(\omega_t^s)\Vert^2\vert \mathcal{F}_t^s]\cr
        &=&\mathbb{E}[\Vert \tilde{\phi}_t^s\Big(\text{grad}f_{I_t^s}(\omega_t^s)-\mathcal{T}_{\omega_0^s}^{\omega_t^s} \big(\text{grad}f_{I_t^s}(\omega_0^s)-\text{grad}f(\omega_0^s)\big) \Big)\cr
        &&+\tilde{\psi}_t^s\Big(\text{grad}f_{I_t^s}(\omega_t^s)-\mathcal{T}_{\omega_{t-1}^s}^{\omega_t^s} \big(\text{grad}f_{I_t^s}(\omega_{t-1}^s)-V_{t-1}^s\big) \Big)-\text{grad}f(\omega_t^s)\Vert^2\vert \mathcal{F}_t^s]\cr
        &=&\mathbb{E}[\Vert\tilde{\phi}_t^s\Big(\text{grad}f_{I_t^s}(\omega_t^s)-\mathcal{T}_{\omega_0^s}^{\omega_t^s} \big(\text{grad}f_{I_t^s}(\omega_0^s)-\text{grad}f(\omega_0^s)\big)-\text{grad}f(\omega_t^s)\Big)\cr
        &&+\tilde{\psi}_t^s\Big(\text{grad}f_{I_t^s}(\omega_t^s)-\mathcal{T}_{\omega_{t-1}^s}^{\omega_t^s} \big(\text{grad}f_{I_t^s}(\omega_{t-1}^s)-V_{t-1}^s\big)-\text{grad}f(\omega_t^s)\Big)\Vert^2\vert \mathcal{F}_t^s]\cr
        &=&\mathbb{E}[\Vert\tilde{\phi}_t^s\Big(\text{grad}f_{I_t^s}(\omega_t^s)-\mathcal{T}_{\omega_0^s}^{\omega_t^s} \big(\text{grad}f_{I_t^s}(\omega_0^s)-\text{grad}f(\omega_0^s)\big)-\text{grad}f(\omega_t^s)\Big)\cr
        &&+\tilde{\psi}_t^s\Big(\text{grad}f_{I_t^s}(\omega_t^s)-\mathcal{T}_{\omega_{t-1}^s}^{\omega_t^s}\text{grad}f_{I_t^s}(\omega_{t-1}^s)-(\text{grad}f(\omega_t^s)-\mathcal{T}_{\omega_{t-1}^s}^{\omega_t^s}\text{grad}f(\omega_{t-1}^s))\Big)\cr
        &&+\tilde{\psi}_t^s\mathcal{T}_{\omega_{t-1}^s}^{\omega_t^s}V_{t-1}^s-\tilde{\psi}_t^s\mathcal{T}_{\omega_{t-1}^s}^{\omega_t^s}\text{grad}f(\omega_{t-1}^s)\Vert^2\vert \mathcal{F}_t^s]\cr
        &=&\mathbb{E}[\Vert\tilde{\phi}_t^s\Big(\text{grad}f_{I_t^s}(\omega_t^s)-\mathcal{T}_{\omega_0^s}^{\omega_t^s} \big(\text{grad}f_{I_t^s}(\omega_0^s)-\text{grad}f(\omega_0^s)\big)-\text{grad}f(\omega_t^s)\Big)\cr
        &&+\tilde{\psi}_t^s\Big(\text{grad}f_{I_t^s}(\omega_t^s)-\mathcal{T}_{\omega_{t-1}^s}^{\omega_t^s}\text{grad}f_{I_t^s}(\omega_{t-1}^s)-(\text{grad}f(\omega_t^s)-\mathcal{T}_{\omega_{t-1}^s}^{\omega_t^s}\text{grad}f(\omega_{t-1}^s))\Big)\Vert^2\vert \mathcal{F}_t^s]\cr
        &&+\mathbb{E}[\Vert\tilde{\psi}_t^s\mathcal{T}_{\omega_{t-1}^s}^{\omega_t^s}V_{t-1}^s-\tilde{\psi}_t^s\mathcal{T}_{\omega_{t-1}^s}^{\omega_t^s}\text{grad}f(\omega_{t-1}^s)\Vert^2\vert \mathcal{F}_t^s]\cr
        &\leq&2(\tilde{\phi}_t^s)^2\mathbb{E}[\Vert\text{grad}f_{I_t^s}(\omega_t^s)-\mathcal{T}_{\omega_0^s}^{\omega_t^s} \big(\text{grad}f_{I_t^s}(\omega_0^s)-\text{grad}f(\omega_0^s)\big)-\text{grad}f(\omega_t^s)\vert \mathcal{F}_t^s]\cr
        &&+2(\tilde{\psi}_t^s)^2\mathbb{E}[\Vert\text{grad}f_{I_t^s}(\omega_t^s)-\mathcal{T}_{\omega_{t-1}^s}^{\omega_t^s}\text{grad}f_{I_t^s}(\omega_{t-1}^s)-(\text{grad}f(\omega_t^s)\cr
        &&-\mathcal{T}_{\omega_{t-1}^s}^{\omega_t^s}\text{grad}f(\omega_{t-1}^s))\Vert^2\vert \mathcal{F}_t^s]+\mathbb{E}[\Vert\tilde{\psi}_t^s\mathcal{T}_{\omega_{t-1}^s}^{\omega_t^s}V_{t-1}^s-\tilde{\psi}_t^s\mathcal{T}_{\omega_{t-1}^s}^{\omega_t^s}\text{grad}f(\omega_{t-1}^s)\Vert^2\vert \mathcal{F}_t^s]
    \end{eqnarray}
    The fourth inequality is due to $\mathbb{E}[(1-\psi_t^s)\text{grad}f_{I_t^s}(\omega_t^s)-\phi_t^s\mathcal{T}_{\omega_0^s}^{\omega_t^s} \big(\text{grad}f_{I_t^s}(\omega_0^s)-\text{grad}f(\omega_0^s)\big)-(1-\psi_t^s)\text{grad}f(\omega_t^s)
    +\tilde{\psi}_t^s\Big(\text{grad}f_{I_t^s}(\omega_t^s)-\mathcal{T}_{\omega_{t-1}^s}^{\omega_t^s}\text{grad}f_{I_t^s}(\omega_{t-1}^s)-(\text{grad}f(\omega_t^s)-\mathcal{T}_{\omega_{t-1}^s}^{\omega_t^s}\text{grad}f(\omega_{t-1}^s))\Big)\vert \mathcal{F}_t^s]=0$
     and $\mathcal{T}_{\omega_{t-1}^s}^{\omega_t^s}V_{t-1}^s-\mathcal{T}_{\omega_{t-1}^s}^{\omega_t^s}\text{grad}f(\omega_{t-1}^s)$ is measurable in $\mathcal{F}_t^s$.
     For the second item at the right side of $\mathrm{(\ref{L7.P.1})}$.
     Similarly the proof of $\mathrm{(\ref{L5.P.1})}$ yields
    \begin{eqnarray}\label{L7.P.2}
        &&\mathbb{E}[\Vert\text{grad}f_{I_t^s}(\omega_t^s)-\mathcal{T}_{\omega_{t-1}^s}^{\omega_t^s}\text{grad}f_{I_t^s}(\omega_{t-1}^s)-(\text{grad}f(\omega_t^s)-\mathcal{T}_{\omega_{t-1}^s}^{\omega_t^s}\text{grad}f(\omega_{t-1}^s))\Vert^2\vert \mathcal{F}_t^s]\cr
        &\leq&\mathbb{E}[\Vert\text{grad}f_{I_t^s}(\omega_t^s)-\mathcal{T}_{\omega_{t-1}^s}^{\omega_t^s}\text{grad}f_{I_t^s}(\omega_{t-1}^s)\Vert^2\vert \mathcal{F}_t^s]\cr
        &\leq & \frac{1}{n}\sum\limits_{i=1}^{n}\Vert\text{grad}f_i(\omega_t^s)-\mathcal{T}_{\omega_{t-1}^s}^{\omega_t^s}\text{grad}f_i(\omega_{t-1}^s)\Vert^2\cr
        &\leq& \frac{2}{n}\sum\limits_{i=1}^{n}\big(\Vert\text{grad}f_i(\omega_t^s)\Vert^2+\Vert\mathcal{T}_{\omega_{t-1}^s}^{\omega_t^s}\text{grad}f_i(\omega_{t-1}^s)\Vert^2\big)\cr
        &\leq& \frac{2}{n}\sum\limits_{i=1}^{n}\big(\Vert\text{grad}f_i(\omega_t^s)\Vert^2+\Vert\text{grad}f_i(\omega_{t-1}^s)\Vert^2\big)\cr
        &\leq& \frac{2}{n}\sum\limits_{i=1}^{n}\big(N^2+N^2\big)\cr
        &=& 4N^2
    \end{eqnarray}
    The first inequality holds due to $\mathbb{E}[\Vert x-\mathbb{E}[x]\Vert^2\vert \mathcal{F}_t^s]\leq \mathbb{E}\Vert[x\Vert^2\vert \mathcal{F}_t^s]$
    Combining the inequalities $\mathrm{(\ref{L6.P.2})}$,$\mathrm{(\ref{L6.P.4})}$ and $\mathrm{(\ref{L7.P.1})}$-$\mathrm{(\ref{L7.P.2})}$ gives
    \begin{eqnarray}\label{L7.P.3}
        &&\mathbb{E}[\Vert V_t^s-\text{grad}f(\omega_t^s)\Vert^2\vert \mathcal{F}_t^s]\cr
        &\leq &4(M^2+\theta^2N^2)(\tilde{\phi}_t^s)^2\Vert \xi_{\omega_0^s}^{\omega_t^s}\Vert^2+8N^2(\tilde{\psi}_t^s)^2\cr
        &&+(\tilde{\psi}_t^s)^2\Vert V_{t-1}^s-\text{grad}f(\omega_{t-1}^s)\Vert^2\cr
        &\leq &4(M^2+\theta^2N^2)(\tilde{\phi}_t^s)^2\Vert \xi_{\omega_0^s}^{\omega_t^s}\Vert^2+8N^2(\psi_t^s)^2\cr
        &&+(\psi_t^s)^2\Vert V_{t-1}^s-\text{grad}f(\omega_{t-1}^s)\Vert^2
    \end{eqnarray}
\end{proof}

\textbf{Proof of Lemma $\mathrm{\ref{L8}}$}
\begin{proof}
    By the definition, we have
\begin{eqnarray}\label{L8.P.1}
    &&\mathbb{E}[\Vert V_t^s-\text{grad}f(\omega_t^s)\Vert^2\vert \mathcal{F}_t^s]\cr
    &=&\mathbb{E}[\Vert \phi_t^s\Big(\text{grad}f_{I_t^s}(\omega_t^s)-\mathcal{T}_{\omega_0^s}^{\omega_t^s} \big(\text{grad}f_{I_t^s}(\omega_0^s)-\text{grad}f(\omega_0^s)\big) \Big)\cr
    &&+\psi_t^s\Big(\text{grad}f_{I_t^s}(\omega_t^s)-\mathcal{T}_{\omega_{t-1}^s}^{\omega_t^s} \big(\text{grad}f_{I_t^s}(\omega_{t-1}^s)-V_{t-1}^s\big) \Big)\cr
    &&+(1-\phi_t^s-\tilde{\psi}_t^s)\text{grad}f_{I_t^s}(\omega_t^s)-\text{grad}f(\omega_t^s)\Vert^2\vert \mathcal{F}_t^s]\cr
    &=&\mathbb{E}[\Vert\phi_t^s\Big(\text{grad}f_{I_t^s}(\omega_t^s)-\mathcal{T}_{\omega_0^s}^{\omega_t^s} \big(\text{grad}f_{I_t^s}(\omega_0^s)-\text{grad}f(\omega_0^s)\big)-\text{grad}f(\omega_t^s)\Big)\cr
    &&+\psi_t^s\Big(\text{grad}f_{I_t^s}(\omega_t^s)-\mathcal{T}_{\omega_{t-1}^s}^{\omega_t^s} \big(\text{grad}f_{I_t^s}(\omega_{t-1}^s)-V_{t-1}^s\big)-\text{grad}f(\omega_t^s)\Big)\cr
    &&+(1-\phi_t^s-\psi_t^s)(\text{grad}f_{I_t^s}(\omega_t^s)-\text{grad}f(\omega_t^s))\Vert^2\vert \mathcal{F}_t^s]\cr
    &=&\mathbb{E}[\Vert\phi_t^s\Big(\text{grad}f_{I_t^s}(\omega_t^s)-\mathcal{T}_{\omega_0^s}^{\omega_t^s} \big(\text{grad}f_{I_t^s}(\omega_0^s)-\text{grad}f(\omega_0^s)\big)-\text{grad}f(\omega_t^s)\Big)\cr
    &&+\psi_t^s\Big(\text{grad}f_{I_t^s}(\omega_t^s)-\mathcal{T}_{\omega_{t-1}^s}^{\omega_t^s}\text{grad}f_{I_t^s}(\omega_{t-1}^s)-(\text{grad}f(\omega_t^s)-\mathcal{T}_{\omega_{t-1}^s}^{\omega_t^s}\text{grad}f(\omega_{t-1}^s))\Big)\cr
    &&+(1-\phi_t^s-\psi_t^s)(\text{grad}f_{I_t^s}(\omega_t^s)-\text{grad}f(\omega_t^s))\cr
    &&+\psi_t^s\mathcal{T}_{\omega_{t-1}^s}^{\omega_t^s}V_{t-1}^s-\psi_t^s\mathcal{T}_{\omega_{t-1}^s}^{\omega_t^s}\text{grad}f(\omega_{t-1}^s)\Vert^2\vert \mathcal{F}_t^s]\cr
    &=&\mathbb{E}[\Vert\phi_t^s\Big(\text{grad}f_{I_t^s}(\omega_t^s)-\mathcal{T}_{\omega_0^s}^{\omega_t^s} \big(\text{grad}f_{I_t^s}(\omega_0^s)-\text{grad}f(\omega_0^s)\big)-\text{grad}f(\omega_t^s)\Big)\cr
    &&+\psi_t^s\Big(\text{grad}f_{I_t^s}(\omega_t^s)-\mathcal{T}_{\omega_{t-1}^s}^{\omega_t^s}\text{grad}f_{I_t^s}(\omega_{t-1}^s)-(\text{grad}f(\omega_t^s)-\mathcal{T}_{\omega_{t-1}^s}^{\omega_t^s}\text{grad}f(\omega_{t-1}^s))\Big)\cr
    &&+(1-\phi_t^s-\psi_t^s)(\text{grad}f_{I_t^s}(\omega_t^s)-\text{grad}f(\omega_t^s))\Vert\cr
    &&+\Vert\psi_t^s\mathcal{T}_{\omega_{t-1}^s}^{\omega_t^s}V_{t-1}^s-\psi_t^s\mathcal{T}_{\omega_{t-1}^s}^{\omega_t^s}\text{grad}f(\omega_{t-1}^s)\Vert^2\cr
    &&+\langle\phi_t^s\Big(\text{grad}f_{I_t^s}(\omega_t^s)-\mathcal{T}_{\omega_0^s}^{\omega_t^s} \big(\text{grad}f_{I_t^s}(\omega_0^s)-\text{grad}f(\omega_0^s)\big)-\text{grad}f(\omega_t^s)\Big)\cr
    &&+\psi_t^s\Big(\text{grad}f_{I_t^s}(\omega_t^s)-\mathcal{T}_{\omega_{t-1}^s}^{\omega_t^s}\text{grad}f_{I_t^s}(\omega_{t-1}^s)-(\text{grad}f(\omega_t^s)-\mathcal{T}_{\omega_{t-1}^s}^{\omega_t^s}\text{grad}f(\omega_{t-1}^s))\Big)\cr
    &&+(1-\phi_t^s-\psi_t^s)(\text{grad}f_{I_t^s}(\omega_t^s)-\text{grad}f(\omega_t^s)),\cr
    &&\psi_t^s\mathcal{T}_{\omega_{t-1}^s}^{\omega_t^s}V_{t-1}^s-\psi_t^s\mathcal{T}_{\omega_{t-1}^s}^{\omega_t^s}\text{grad}f(\omega_{t-1}^s)\rangle\vert \mathcal{F}_t^s]\cr
    &=&\mathbb{E}[\Vert\phi_t^s\Big(\text{grad}f_{I_t^s}(\omega_t^s)-\mathcal{T}_{\omega_0^s}^{\omega_t^s} \big(\text{grad}f_{I_t^s}(\omega_0^s)-\text{grad}f(\omega_0^s)\big)-\text{grad}f(\omega_t^s)\Big)\cr
    &&+\psi_t^s\Big(\text{grad}f_{I_t^s}(\omega_t^s)-\mathcal{T}_{\omega_{t-1}^s}^{\omega_t^s}\text{grad}f_{I_t^s}(\omega_{t-1}^s)-(\text{grad}f(\omega_t^s)-\mathcal{T}_{\omega_{t-1}^s}^{\omega_t^s}\text{grad}f(\omega_{t-1}^s))\Big)\cr
    &&+(1-\phi_t^s-\psi_t^s)(\text{grad}f_{I_t^s}(\omega_t^s)-\text{grad}f(\omega_t^s))\Vert\cr
    &&+\Vert\psi_t^s\mathcal{T}_{\omega_{t-1}^s}^{\omega_t^s}V_{t-1}^s-\psi_t^s\mathcal{T}_{\omega_{t-1}^s}^{\omega_t^s}\text{grad}f(\omega_{t-1}^s)\Vert^2\vert \mathcal{F}_t^s]\cr
    &\leq&3(\phi_t^s)^2\mathbb{E}[\Vert\text{grad}f_{I_t^s}(\omega_t^s)-\mathcal{T}_{\omega_0^s}^{\omega_t^s} \big(\text{grad}f_{I_t^s}(\omega_0^s)-\text{grad}f(\omega_0^s)\big)-\text{grad}f(\omega_t^s)\Vert^2\vert \mathcal{F}_t^s]\cr
    &&+3(\psi_t^s)^2\mathbb{E}[\Vert\text{grad}f_{I_t^s}(\omega_t^s)-\mathcal{T}_{\omega_{t-1}^s}^{\omega_t^s}\text{grad}f_{I_t^s}(\omega_{t-1}^s)-(\text{grad}f(\omega_t^s)\cr
    &&-\mathcal{T}_{\omega_{t-1}^s}^{\omega_t^s}\text{grad}f(\omega_{t-1}^s))\Vert^2\vert \mathcal{F}_t^s]+3(1-\phi_t^s-\psi_t^s)^2\mathbb{E}[\Vert(\text{grad}f_{I_t^s}(\omega_t^s)-\text{grad}f(\omega_t^s))\Vert^2\vert \mathcal{F}_t^s]\cr
    &&+(\psi_t^s)^2\mathbb{E}[\Vert\mathcal{T}_{\omega_{t-1}^s}^{\omega_t^s}V_{t-1}^s-\mathcal{T}_{\omega_{t-1}^s}^{\omega_t^s}\text{grad}f(\omega_{t-1}^s)\Vert^2\vert \mathcal{F}_t^s]
\end{eqnarray}
The fifth equality is based on $\mathbb{E}[\langle\phi_t^s\Big(\text{grad}f_{I_t^s}(\omega_t^s)-\mathcal{T}_{\omega_0^s}^{\omega_t^s} \big(\text{grad}f_{I_t^s}(\omega_0^s)-\text{grad}f(\omega_0^s)\big)-\text{grad}f(\omega_t^s)\Big)
+\psi_t^s\Big(\text{grad}f_{I_t^s}(\omega_t^s)-\mathcal{T}_{\omega_{t-1}^s}^{\omega_t^s}\text{grad}f_{I_t^s}(\omega_{t-1}^s)-(\text{grad}f(\omega_t^s)-\mathcal{T}_{\omega_{t-1}^s}^{\omega_t^s}\text{grad}f(\omega_{t-1}^s))\Big)
+(1-\phi_t^s-\psi_t^s)(\text{grad}f_{I_t^s}(\omega_t^s)-\text{grad}f(\omega_t^s)),
\psi_t^s\mathcal{T}_{\omega_{t-1}^s}^{\omega_t^s}V_{t-1}^s-\psi_t^s\mathcal{T}_{\omega_{t-1}^s}^{\omega_t^s}\text{grad}f(\omega_{t-1}^s)\rangle\vert \mathcal{F}_t^s]=0$
 and $\mathcal{T}_{\omega_{t-1}^s}^{\omega_t^s}V_{t-1}^s-\mathcal{T}_{\omega_{t-1}^s}^{\omega_t^s}\text{grad}f(\omega_{t-1}^s)$ is measurable in $\mathcal{F}_t^s$.
 Now, we consider the each term on the right side of $\mathrm{(\ref{L8.P.1})}$.
 By lemma $\mathrm{\ref{L5}}$, we get
\begin{eqnarray}\label{L8.P.2}
        &&\mathbb{E}[\Vert\text{grad}f_{I_t^s}(\omega_t^s)\mathcal{T}_{\omega_{t-1}^s}^{\omega_t^s}\text{grad}f_{I_t^s}(\omega_{t-1}^s)-(\text{grad}f(\omega_t^s)-\mathcal{T}_{\omega_{t-1}^s}^{\omega_t^s}\text{grad}f(\omega_{t-1}^s))\Vert^2\vert \mathcal{F}_t^s]\cr
        &\leq&\mathbb{E}[\Vert\text{grad}f_{I_t^s}(\omega_t^s)-\mathcal{T}_{\omega_{t-1}^s}^{\omega_t^s}\text{grad}f_{I_t^s}(\omega_{t-1}^s)\Vert^2\vert \mathcal{F}_t^s]\cr
        &\leq &2(M^2+\theta^2N^2)\Vert \xi_{\omega_1}^{\omega_2}\Vert^2
\end{eqnarray}
Combining the inequalities $\mathrm{(\ref{L6.P.2})}$,$\mathrm{(\ref{L6.P.4})}$,$\mathrm{(\ref{L7.P.2})}$ and $\mathrm{(\ref{L8.P.2})}$, we have
\begin{eqnarray}\label{L8.P.3}
    &&\mathbb{E}[\Vert V_t^s-\text{grad}f(\omega_t^s)\Vert^2\vert \mathcal{F}_t^s]\cr
    &\leq&6(\phi_t^s)^2(M^2+\theta^2N^2)\Vert \xi_{\omega_0^s}^{\omega_t^s}\Vert^2+6(\psi_t^s)^2(M^2+\theta^2N^2)\Vert \xi_{\omega_1}^{\omega_2}\Vert^2\cr
    &&+12(1-\phi_t^s-\psi_t^s)^2N^2+(\psi_t^s)^2\Vert V_{t-1}^s-\text{grad}f(\omega_{t-1}^s)\Vert^2
\end{eqnarray}
\end{proof}

\noindent
\textbf{Proof of Lemma $\mathrm{\ref{L9}}$}
\begin{proof}
    From assumption 2, if $\Vert \xi_{\omega_0^s}^{\omega_t^s}\Vert> \rho$, then $\omega_t^s=R_{\omega_0^s}(\xi_{\omega_0^s}^{\omega_t^s}) \notin \Omega$.
     Thit is contradicted with assumption 2 $\omega_t^s\in \Omega$.
    Hence, $\Vert \xi_{\omega_0^s}^{\omega_t^s}\Vert^2\leq \rho^2$, substituting this result into lemma $\mathrm{\ref{L8}}$ and taking the mathematical
    expectation, we have
    \begin{eqnarray}\label{L9.P.1}
        &&\mathbb{E}[\Vert V_t^s-\text{grad}f(\omega_t^s)\Vert^2]\cr
        &\leq& 6(M^2+\theta^2N^2)\big((\phi_t^s)^2\rho^2+(\psi_t^s)^2(\alpha_{t-1}^s)^2\mathbb{E}[\Vert V_{t-1}^s\Vert^2)]\cr
        &&+12(1-\phi_t^s-\psi_t^s)^2N^2+(\psi_t^s)^2\mathbb{E}[\Vert V_{t-1}^s-\text{grad}f(\omega_{t-1}^s)\Vert^2]\cr
        &\leq& 6(M^2+\theta^2N^2)(\phi_t^s)^2\rho^2+12(M^2+\theta^2N^2)(\psi_t^s)^2(\alpha_{t-1}^s)^2 \mathbb{E}[\Vert \text{grad}f(\omega_{t-1}^s)\Vert^2]\cr
        &&+\big[1+12(M^2+\theta^2N^2)(\alpha_{t-1}^s)^2  \big](\psi_t^s)^2\mathbb{E}[\Vert V_{t-1}^s-\text{grad}f(\omega_{t-1}^s)\Vert^2]\cr
        &&+12(1-\phi_t^s-\psi_t^s)^2N^2
    \end{eqnarray}
    The last equality is due to $(a+b)^2\leq 2a^2+2b^2$. For any $2\leq t \leq m-1$, we have
    \begin{eqnarray}\label{L9.P.2}
        &&\frac{\mathbb{E}[\Vert V_t^s-\text{grad}f(\omega_t^s)\Vert^2]}{\alpha_{t-1}^s}- \frac{\mathbb{E}[\Vert V_{t-1}^s-\text{grad}f(\omega_{t-1}^s)\Vert^2]}{\alpha_{t-2}^s}\cr
        &\leq &\frac{6(M^2+\theta^2N^2)}{\alpha_{t-1}^s}(\phi_t^s)^2\rho^2+12(M^2+\theta^2N^2)(\psi_t^s)^2\alpha_{t-1}^s \mathbb{E}[\Vert \text{grad}f(\omega_{t-1}^s)\Vert^2]\cr
        &&+\big[-\frac{1}{\alpha_{t-2}^s} +\frac{1}{\alpha_{t-1}^s}+12(M^2+\theta^2N^2)\alpha_{t-1}^s  \big](\psi_t^s)^2\mathbb{E}[\Vert V_{t-1}^s-\text{grad}f(\omega_{t-1}^s)\Vert^2]\cr
        &&+\frac{12(1-\phi_t^s-\psi_t^s)^2N^2}{\alpha_{t-1}^s}\cr
        &\leq &\frac{6(M^2+\theta^2N^2)}{\alpha_{t-1}^s}(\phi_t^s)^2\rho^2+12(M^2+\theta^2N^2)(\psi_t^s)^2\alpha_{t-1}^s \mathbb{E}[\Vert \text{grad}f(\omega_{t-1}^s)\Vert^2]\cr
        &&+\big[-\frac{1}{\alpha_{t-2}^s} +\frac{\psi_t^s}{\alpha_{t-1}^s}+12(M^2+\theta^2N^2)\alpha_{t-1}^s \big]\mathbb{E}[\Vert V_{t-1}^s-\text{grad}f(\omega_{t-1}^s)\Vert^2]\cr
        &&+\frac{12(1-\phi_t^s-\psi_t^s)^2N^2}{\alpha_{t-1}^s}
    \end{eqnarray}
    The second inequality is due to $(\psi_t^s)^2\leq 1$ and $(\psi_t^s)^2\leq \psi_t^s$.
     Now, we consider the third term on the right side of $\mathrm{(\ref{L9.P.2})}$
    \begin{equation}\label{L9.P.3}
        -\frac{1}{\alpha_{t-2}^s} +\frac{\psi_t^s}{\alpha_{t-1}^s}+12(M^2+\theta^2N^2)\alpha_{t-1}^s=\frac{1}{\alpha_{t-1}^s}-\frac{1}{\alpha_{t-2}^s} -\frac{1-\psi_t^s}{\alpha_{t-1}^s}+12(M^2+\theta^2N^2)\alpha_{t-1}^s
    \end{equation}
    Giving by conditions $\alpha_t^s=(t+s+\kappa+2)^{-P}C_\alpha$ and $\psi_t^s=1-(t+s+\kappa+1)^{-Q}C_\psi$, then $\frac{1-\psi_t^s}{\alpha_{t-1}^s}=\frac{C_\psi}{C_\alpha}\cdot (t+s+\kappa+1)^{P-Q}$, and
    $\frac{1}{\alpha_{t-1}^s}-\frac{1}{\alpha_{t-2}^s}=\frac{1}{C_\alpha}\big((t+s+\kappa+1)^P-(t+s+\kappa)^P\big)$. Let $g(x)=(x+\kappa)^P,0<P<1$, it is easily verified that $g''(x)=P(P-1)(x+\kappa)^{P-2}$, and $g''(x)<0$, $x>0$. 
    Thus $g(x+1)\leq g(x)+g'(x)$, i.e., $(t+s+\kappa+1)^P-(t+s+\kappa)^P\leq P(t+s+\kappa)^{P-1}$. Substituting this result into $\mathrm{(\ref{L9.P.3})}$, then we have
    \begin{eqnarray}\label{L9.P.4}
        &&\frac{1}{\alpha_{t-1}^s}-\frac{1}{\alpha_{t-2}^s} -\frac{1-\psi_t^s}{\alpha_{t-1}^s}+12(M^2+\theta^2N^2)\alpha_{t-1}^s\cr
        &\leq& \frac{P}{C_\alpha}(t+s+\kappa)^{P-1}-\frac{C_\psi}{C_\alpha}\cdot (t+s+\kappa+1)^{P-Q}+12(M^2+\theta^2N^2)\alpha_{t-1}^s
    \end{eqnarray}
    By conditions $P\geq \frac{\gamma Q-1}{\gamma-1}$, $\frac{1-P}{Q-P}\geq \gamma >1$ and the definition of $\kappa$. If $x\geq \kappa$, it is known that $0<x+1\leq x^{\gamma} \leq x^{\frac{1-P}{Q-P}}$. 
    Note that $0<Q-P<Q<1$, the function $y=x^{Q-P}$ is a monotonically increasing in $x\geq 1$. We can easily verify that $0<(x+1)^{Q-P}\leq (x^{\frac{1-P}{Q-P}})^{Q-P}=x^{1-P}$. Therefore,
    $\frac{1}{x^{1-P}}\leq \frac{1}{(x+1)^{Q-P}}$, i.e., $x^{P-1}\leq (x+1)^{P-Q},x\geq \kappa$, implies that $(t+s+\kappa)^{P-1}\leq (t+s+\kappa+1)^{P-Q}$. This together with $\mathrm{(\ref{L9.P.4})}$ leads to
    \begin{eqnarray}\label{L9.P.5}
        &&\frac{1}{\alpha_{t-1}^s}-\frac{1}{\alpha_{t-2}^s} -\frac{1-\psi_t^s}{\alpha_{t-1}^s}+12(M^2+\theta^2N^2)\alpha_{t-1}^s\cr
        &\leq& \frac{P-C_\psi}{C_\alpha}\cdot (t+s+\kappa+1)^{P-Q}+12(M^2+\theta^2N^2)\alpha_{t-1}^s\cr
        &\leq&-6C_\alpha\beta(M^2+\theta^2N^2)\cdot (t+s+\kappa+1)^{P-Q}+12(M^2+\theta^2N^2)\alpha_{t-1}^s\cr
        &=&6(M^2+\theta^2N^2)C_\alpha(-\beta (t+s+\kappa+1)^{P-Q}+2(t+s+\kappa+1)^{-P})
    \end{eqnarray}
    Noting that $0<a<1$, $y=a^x$ is a monotonically decreasing function in $x>0$, hence, 
    $y=(\frac{1}{t+s+\kappa+1})^x$ is a monotonically decreasing function in $x>0$.
    By condition $\frac{Q}{2}\leq P\leq Q$, we get
    $\big(\frac{1}{t+s+\kappa+1}\big)^{Q-P}\geq \big(\frac{1}{t+s+\kappa+1}\big)^{P}$,
     i.e., $(t+s+\kappa+1)^{P-Q}\geq (t+s+\kappa+1)^{-P}$.
     Substituting this result back to$\mathrm{(\ref{L9.P.5})}$, we have
    \begin{eqnarray}\label{L9.P.6}
        &&\frac{1}{\alpha_{t-1}^s}-\frac{1}{\alpha_{t-2}^s} -\frac{1-\psi_t^s}{\alpha_{t-1}^s}+\frac{4(M^2+\theta^2N^2)}{b}\alpha_{t-1}^s\cr
        &\leq& 6(M^2+\theta^2N^2)C_\alpha(-\beta (t+s+\kappa+1)^{P-Q}+2(t+s+\kappa+1)^{-P})\cr
        &\leq& 6(M^2+\theta^2N^2)C_\alpha(-\beta(t+s+\kappa+1)^{-P}+2(t+s+\kappa+1)^{-P})\cr
        &=& 6(2-\beta)(M^2+\theta^2N^2)\alpha_{t-1}^s
    \end{eqnarray}
    Combining the inequalities $\mathrm{(\ref{L9.P.2})}$ and $\mathrm{(\ref{L9.P.6})}$, we obtain
    \begin{eqnarray}\label{L9.P.7}
        &&\frac{\mathbb{E}[\Vert V_t^s-\text{grad}f(\omega_t^s)\Vert^2]}{\alpha_{t-1}^s}- \frac{\mathbb{E}[\Vert V_{t-1}^s-\text{grad}f(\omega_{t-1}^s)\Vert^2]}{\alpha_{t-2}^s}\cr
        &\leq &\frac{6(M^2+\theta^2N^2)}{\alpha_{t-1}^s}(\phi_t^s)^2\rho^2+12(M^2+\theta^2N^2)(\psi_t^s)^2\alpha_{t-1}^s \mathbb{E}[\Vert \text{grad}f(\omega_{t-1}^s)\Vert^2]\cr
        &&+\big[-\frac{1}{\alpha_{t-2}^s} +\frac{\psi_t^s}{\alpha_{t-1}^s}+12(M^2+\theta^2N^2)\alpha_{t-1}^s \big]\mathbb{E}[\Vert V_{t-1}^s-\text{grad}f(\omega_{t-1}^s)\Vert^2]\cr
        &&+\frac{12(1-\phi_t^s-\psi_t^s)^2N^2}{\alpha_{t-1}^s}\cr
        &\leq & \frac{6(M^2+\theta^2N^2)}{\alpha_{t-1}^s}(\phi_t^s)^2\rho^2+12(M^2+\theta^2N^2)\alpha_{t-1}^s \mathbb{E}[\Vert \text{grad}f(\omega_{t-1}^s)\Vert^2]\cr
        &&+6(2-\beta)(M^2+\theta^2N^2)\alpha_{t-1}^s\mathbb{E}[\Vert V_{t-1}^s-\text{grad}f(\omega_{t-1}^s)\Vert^2]\cr
        &&+\frac{12(1-\psi_t^s)^2N^2}{\alpha_{t-1}^s}
    \end{eqnarray}
\end{proof}



\end{appendices}






\renewcommand\refname{References}

\end{document}